\documentclass[a4paper,12pt]{article}

\usepackage{cite} 
\usepackage{authblk}
\usepackage{parskip}
\usepackage{mathtools}
\usepackage{amssymb}
\usepackage{amsmath}
\usepackage{latexsym}
\usepackage{mathrsfs}  
\usepackage{bbm}       
\usepackage{amsthm}    
\usepackage{relsize}   
\usepackage{graphicx}
\usepackage{thmtools}  
\usepackage{mathrsfs}  
\usepackage{paralist}  
\usepackage{lipsum}    
\usepackage[all]{xy}   
\usepackage{hyperref}

\numberwithin{equation}{section}

\usepackage{titlesec}   

\titleformat{\section}[block]{\bfseries\filcenter}
{{\upshape\thesection\enspace}}{.5em}{}

\titleformat{\subsection}[block]{\filcenter}
{{\upshape\thesubsection\enspace}}{.5em}{} 

\titleformat{\subsubsection}[block]{\filcenter}
{{\upshape\thesubsubsection\enspace}}{.5em}{} 

\usepackage{enumitem}  
\setlist{nosep}  
\setitemize[0]{leftmargin=*}
\setenumerate[0]{leftmargin=*}
\setenumerate[1]{label={(\arabic*)}} 

\newcommand{\N}{\mathbb{N}}     
\newcommand{\Q}{\mathbb{Q}}     
\newcommand{\R}{\mathbb{R}}     
\newcommand{\Prob}{\mathbb{P}}  
\newcommand{\Exp}{\mathbb{E}}   
\newcommand{\goth}[1]{\mathfrak{#1}} 
\newcommand{\ind}[2]{\mathbbm{1}_{#1}\left( #2 \right)}          
\newcommand{\inner}[2]{\left\langle #1 \, , \, #2 \right\rangle} 
\newcommand{\norm}[1]{\left|\left|#1\right|\right|}              
\newcommand{\triplet}[3]{\left( #1, #2, #3 \right) }             
\newcommand{\ProbSpace}{\triplet{\Omega}{\mathscr{F}}{\Prob}}    
\newcommand{\quadraVari}[1]{\left\langle \left\langle  #1  \right\rangle \right\rangle } 
\newcommand{\abs}[1]{\left| #1 \right|}                          
\newcommand{\defeq}{\mathrel{\mathop:}=}                         
\newcommand\restr[2]{{
  \left.\kern-\nulldelimiterspace 
  #1 
  \vphantom{\big|} 
  \right|_{#2} 
  }}
  
\makeatletter
\newsavebox{\@brx}
\newcommand{\llangle}[1][]{\savebox{\@brx}{\(\m@th{#1\langle}\)}%
  \mathopen{\copy\@brx\kern-0.5\wd\@brx\usebox{\@brx}}}
\newcommand{\rrangle}[1][]{\savebox{\@brx}{\(\m@th{#1\rangle}\)}%
  \mathclose{\copy\@brx\kern-0.5\wd\@brx\usebox{\@brx}}}
\makeatother

\theoremstyle{plain} 
\newtheorem{theorem}{Theorem}[section]    
\newtheorem{proposition}[theorem]{Proposition} 

\newtheorem{lemma}[theorem]{Lemma}
\newtheorem{assumption}[theorem]{Assumption}

\theoremstyle{definition} 
\newtheorem{definition}[theorem]{Definition}
\newtheorem{example}[theorem]{Example}
\newtheorem{remark}[theorem]{Remark}

 \title{It\^{o}'s formula for L\'evy-It\^{o} processes taking values in the dual of nuclear space}
\author{C. A. Fonseca-Mora}
\affil{  Escuela de Matem\'{a}tica,  Universidad de Costa Rica,\\ San Jos\'{e}, 11501-2060, Costa Rica. \\

\noindent E-mail:  christianandres.fonseca@ucr.ac.cr }

\date{}

\begin{document}

 \maketitle

\abstract{ 
Using the theory of stochastic integration in duals of nuclear spaces with respect to cylindrical martingale-valued measures, a vector-valued It\^{o} formula is proved for generalized It\^{o} processes defined with respect to these integrals. The abstract result is then applied to prove an It\^{o} formula for L\'evy-It\^{o} processes defined with respect to L\'evy processes taking values in the dual of a reflexive nuclear space.  
}

\smallskip



\section{Introduction}

L\'evy processes are some of the most widely used models for the noise of a stochastic differential equation (SDE) or a stochastic partial differential equation (SPDE), providing a natural framework for modeling random phenomena that exhibits both continuous fluctuations and discontinuous jumps. While the theory employing L\'evy processes in finite-dimensional spaces is now well established, many applications arising in SDEs and SPDEs, mathematical physics, finance and dynamical systems require a corresponding theory in function spaces. In particular, some stochastic models naturally lead to random objects that cannot be described as ordinary functions, but must be interpreted as generalized functions or distributions. Classical examples include white noise models, random fields and solutions to SPDEs driven by singular forcing terms. In this framework, spaces of distributions, or more generally, the dual of a nuclear space,  provide a natural state space for stochastic dynamics (e.g. see \cite{BojdeckiGorostiza:1986, HitsudaMitoma:1986, KallianpurMitoma:1992, KallianpurXiong, KorezliogluMartias:1984}). 

The study of L\'evy processes in duals of nuclear spaces initiated with the work of \"Ust\"unel in \cite{Ustunel:1984}, proving for the first time a L\'evy-It\^{o} decomposition for additive processes taking values in a space of distributions. An extension to the work of \"Ust\"unel is carried out by the author in \cite{FonsecaMora:Levy}. There, the author initiated a systematic study of L\'evy processes in the strong dual $\Phi'$ of a general nuclear space $\Phi$. In particular, L\'evy measures where thoroughly studied and it is proved a L\'evy-It\^{o} decomposition with similar properties as in the finite dimensional setting. This allowed the author to develop a theory of stochastic integration with respect to square integrable L\'evy processes in \cite{FonsecaMora:SPDECMVM} and to prove existence and uniqueness of solutions to SPDEs with multiplicative noise. In \cite{FonsecaMora:SPDELevy}, the author carries out a further extension of the theory of stochastic integration and of (time-inhomogeneous) SPDEs to the case where the driving noise if a general L\'evy process (i.e. without finite moments assumptions).

A natural next step in the development of stochastic analysis with L\'evy noise in the dual of a nuclear space, is to establish an It\^{o} formula. 
In the literature, there are some works that prove generalizations of It\^{o} formula in a space of distributions for $\R^d$-valued semimartigales (see e.g.\cite{Bhar:2017, CatuognoOlivera:2014, Krylov:2011,  Ustunel:1982Ito}). However, to the extent of our knowledge, the only work that proves It\^{o}'s formula in a space of distributions with respect to  a distribution-valued (non-homogeneous) Wiener processes is the work of Jakubowski in \cite{Jakubowski:1995}. 

In this article we prove a vector-valued It\^{o} formula with respect to a L\'evy-It\^{o} process taking values in the dual of a reflexive nuclear space. To do this, we follow a similar program to that in \cite{FonsecaMora:SPDECMVM, FonsecaMora:SPDELevy}. First, we prove an It\^{o} formula for an It\^{o} process defined with respect to a cylindrical martingale-valued measure. Then, we apply our result to prove an It\^{o} formula for a L\'evy process via its L\'evy-It\^{o} decomposition obtained in \cite{FonsecaMora:Levy}. In the following paragraphs we give a description of our main results. 

Let $\Phi$ be a locally convex space with strong dual $\Phi'$. Let $M=(M(t,A): t \geq 0, A \in \mathcal{R})$ be a cylindrical martingale-valued measure on $\Phi'$, i.e. $(M(t,A): t \geq 0)$ is a cylindrical martingale in $\Phi'$ for each $A \in \mathcal{R}$ and $M(t,\cdot)$ is finitely additive on a $\mathcal{R}$ for each $t \geq 0$. Here $\mathcal{R} \subseteq \mathcal{B}(U)$ is a ring and $U$ a topological space. Let $\Psi$ and $\Gamma$ be two separable, complete, bornological, nuclear spaces (as for example are the classical spaces of distributions). A $\Psi'$-valued generalized It\^o process is of the form 
$$X_{t}= \xi+ \int_{0}^{t} B(r)dr+ \int_{0}^{t} \int_{U} R(r,u) M(dr,du), \quad t \geq 0, $$
where further details on the processes $B: \R_{+} \times \Omega \rightarrow \Psi'$ and $R: \R_{+} \times \Omega \times U \rightarrow \mathcal{L}(\Phi',\Psi')$ will be given later. For a $C^{1,2}$ class function $F:[0, \infty) \times \Psi' \rightarrow \Gamma'$, we will show that almost surely we have for every $t\geq 0$
$$
\begin{aligned}
& F(t, X_t) = F(0,\xi)+\int_{0}^{t} F_{t}(s, X_{s-}) d s + \int_{0}^{t}\!\!\int_U \, F_{x}(s, X_{s-})R(s, u) \, M(ds, du) \\
& + \int_{0}^{t} \, F_{x}(s, X_{s-}) B(s) \, ds  +\frac{1}{2} \int_{0}^{t} \!\!\int_U \textup{Tr}_{R(s, u)} \left( F_{x x}(s, X_{s-}) \right) \mu(du) dr \\
& +\sum_{0<s \leq t}\left[F\left(s,X_s\right)-F\left(s,X_{s-}\right)-F_x\left(s,X_{s-}\right)\left(\Delta X_{s} \right)\right],
\end{aligned}
$$
where $\mu$ is a $\sigma$-finite measure on $(U, \mathcal{B}(U))$ satisfying $\mu(A)< \infty$ $\forall \, A \in \mathcal{R}$, and related to the covariance structure of $M$ (see \eqref{eqDefiNuclearCMVM}). The above It\^{o} formula generalizes that in \cite{Jakubowski:1995} in several directions. First, the (non-homogeneous) Wiener processes are a particular example of a cylindrical martingale-valued measure (see Example 3.4 in \cite{FonsecaMora:SPDECMVM}). Second, each $M(t,A)$ is a cylindrical martingale with discontinuous paths, hence $M$ is not necessarily a $\Phi'$-valued process and $\Phi$ is only assumed to be a locally convex space. Third, the formula in 
\cite{Jakubowski:1995} is formulated for functions $F:[0, \infty) \times \Psi' \rightarrow H$, where $H$ is a separable Hilbert space. This formulation is natural if the driven noise is Wiener. However, for a more general cylindrical noise this is an unnecessary restriction. For our result on It\^o's formula we found more convenient to formulate it assuming the image of $F$ is the dual $\Psi'$ of the complete, bornological, nuclear space $\Psi$ (in particular, $\Psi'$ can be any of the classical spaces of distributions). 

For our proof of  It\^{o}'s formula with respect to $M$, we have employed a generalization of the arguments used in \cite{Jakubowski:1995} to allows jumps. The main idea is to use an appropriate localizing sequence to settle the proof in the context of a separable Hilbert space embedded in $\Psi'$, and then to employ  It\^{o}'s formula for cylindrical martingale-valued measures in Hilbert spaces recently proved in \cite{CambroneroEtAl:ItoFormula}. 

Now we describe our main result on It\^{o} formula for L\'evy-It\^o processes. 
Assume now that $\Phi$ is a reflexive nuclear space and that $(L_{t}:t \geq 0)$ is a $\Phi'$-valued L\'evy process with L\'evy -It\^o decomposition 
$$ L_{t}=t\goth{m}+W_{t}+\int_{B_{\rho'}(1)} f \widetilde{N} (t,df)+\int_{B_{\rho'}(1)^{c}} f N (t,df), $$
with Poisson random measure $N$ and corresponding L\'evy measure $\nu$ (see Section \ref{sectItoFormLevyProcess} for further details on the components of the above decomposition). A L\'{e}vy-It\^{o} process defined by $L$ is a $\Psi'$-valued process of the form
\begin{eqnarray*}
X_t 
& = &\xi + \int_0^t B(s) \, ds +\int_0^t R(s,0)\goth{m} \, ds + \int_{0}^{t}  R(r,0) dW_{r} \\ 
& + & \int_{0}^{t} \int_{B_{\rho'}(1)}  R(r,f) \widetilde{N}(dr,df) + \int_0^t\!\!\int_{B_{\rho'}(1)^c} R(s,u) u \, N(ds, du),     
\end{eqnarray*}
where $R: [0, \infty) \times \Omega \times \Phi' \rightarrow \mathcal{L}(\Phi',\Psi')$. The first integral defined via radonification, the second and third using a cylindrical martingale-valued measure, and the fourth is a finite random sum. Using our It\^{o} formula for cylindrical martingale-valued measures, and an interlacing procedure, we prove the following version of It\^{o} formula for the L\'{e}vy-It\^{o} process $X_{t}$, 
\begin{flalign*}
& F(t, X_t) \\
& = F(0,\xi)+\int_{0}^{t} F_{t}(s, X_{s-}) ds+\int_0^t F_x(s, X_{s-}) R(s,0) \, dW_s \\
& +  \int_{0}^{t} F_x(s, X_{s-}) (B(s)+R(s,0) \goth{m}) \, ds   + \frac{1}{2} \int_{0}^{t}  \mbox{Tr}_{R(s, 0)} \left( F_{x x}(s, X_{s-})\right)  ds  \\ 
   & + \int_{0}^{t}\!\! \int_{B_{\rho'}(1)^c} \left[ F(s,X_{s-}+R(s,f)f) - F(s,X_{s-}) \right] \, N(ds,df) \\
   & +  \int_{0}^{t}\!\! \int_{B_{\rho'}(1)} \left[ F(s,X_{s-}+R(s,f)f) - F(s,X_{s-}) \right] \, \tilde{N}(ds,df) \\
   & +  \int_{0}^{t}\!\! \int_{B_{\rho'}(1)} \left[ F(s,X_{s-}+R(s,f)f) - F(s,X_{s-})-F_{x}(s,X_{s-})(R(s,f)f) \right] \,\nu(df)ds.
\end{flalign*}
The above formula generalizes that in the finite-dimensional spaces (see e.g. \cite{ApplebaumLPSC}) to the context of spaces of distributions (or more generally, duals of reflexive nuclear spaces).  

The organization of the paper is the following. In Section \ref{sectPrelim} we list some important notions on nuclear spaces and their duals, properties of cylindrical and
stochastic processes,  and differentiation in the dual of a nuclear space. In Section \ref{secContAndDisconCMVM} we recall basic properties of cylindrical martingale-valued measures and introduce some new results on its decomposition in continuous and purely discontinuous parts. Section \ref{sectionGeneItoProcesses} is devolted to the construction of the generalized It\^{o} processess defined with respect to a cylindrical martingale-valued measure. The corresponding It\^{o}'s formula for these processes is proved in Section \ref{sectItoFormulaCMVM}. The proof of  It\^{o}'s formula  for L\'{e}vy-It\^{o} processes taking values in the dual of a nuclear space is carried out in Section \ref{sectItoFormLevyProcess}. Finally, in Section \ref{sectExamples} we give some practical examples of our definition of L\'evy-It\^{o} process and It\^{o}'s formula.

\section{Preliminaries}\label{sectPrelim}

\subsection{Locally convex spaces and linear operators}\label{sectionLCS}

Let $\Phi$ be a locally convex space (we will only consider vector spaces over $\R$). $\Phi$ is \emph{quasi-complete} if each bounded and closed subset of it is complete. $\Phi$ is called  \emph{bornological} (respectively \emph{ultrabornological}) if it is the inductive limit of a family of normed (respectively Banach)  spaces. A \emph{barreled space} is a locally convex space such that every convex, balanced, absorbing and closed subset is a neighborhood of zero. For further details see \cite{Jarchow, Schaefer, Treves}.   

If $p$ is a continuous semi-norm on $\Phi$ and $r>0$, the closed ball of radius $r$ of $p$ given by $B_{p}(r) = \left\{ \phi \in \Phi: p(\phi) \leq r \right\}$ is a closed, convex, balanced neighborhood of zero in $\Phi$. A continuous seminorm (respectively norm) $p$ on $\Phi$ is called \emph{Hilbertian} if $p(\phi)^{2}=Q(\phi,\phi)$, for all $\phi \in \Phi$, where $Q$ is a symmetric, non-negative bilinear form (respectively inner product) on $\Phi \times \Phi$. For any given continuous seminorm $p$ on $\Phi$ let $\Phi_{p}$ be the Banach space that corresponds to the completion of the normed space $(\Phi / \mbox{ker}(p), \tilde{p})$, where $\tilde{p}(\phi+\mbox{ker}(p))=p(\phi)$ for each $\phi \in \Phi$. If $\Phi_{p}$ is separable then we say that $p$ is \emph{separable}.
We denote by  $\Phi'_{p}$ the Banach space dual to $\Phi_{p}$ and by $p'$ the corresponding dual norm. Observe that if $p$ is Hilbertian then $\Phi_{p}$ and $\Phi'_{p}$ are Hilbert spaces.

 The canonical inclusion from $\Phi$ into $(\Phi / \mbox{ker}(p), \tilde{p})$ has a continuous linear extension that we denote by $i_p$. If $q$ is another continuous seminorm on $\Phi$ for which $p \leq q$, we have that $\mbox{ker}(q) \subseteq \mbox{ker}(p)$ and the inclusion map from $\Phi / \mbox{ker}(q)$ into $\Phi / \mbox{ker}(p)$ has a unique continuous and linear extension that we denote by $i_{p,q}:\Phi_{q} \rightarrow \Phi_{p}$. Furthermore, we have the following relation: $i_{p}=i_{p,q} \circ i_{q}$.

Let $p$ and $q$ be continuous Hilbertian semi-norms on $\Phi$ such that $p \leq q$.
The space of continuous linear operators (respectively Hilbert-Schmidt operators) from $\Phi_{q}$ into $\Phi_{p}$ is denoted by $\mathcal{L}(\Phi_{q},\Phi_{p})$ (respectively $\mathcal{L}_{2}(\Phi_{q},\Phi_{p})$) and the operator norm (respectively Hilbert-Schmidt norm) is denote by $\norm{\cdot}_{\mathcal{L}(\Phi_{q},\Phi_{p})}$ (respectively $\norm{\cdot}_{\mathcal{L}_{2}(\Phi_{q},\Phi_{p})}$).

We denote by $\Phi'$ the topological dual of $\Phi$ and by $\inner{f}{\phi}$ the canonical pairing of elements $f \in \Phi'$, $\phi \in \Phi$. Unless otherwise specified, $\Phi'$ will always be consider equipped with its \emph{strong topology}, i.e. the topology on $\Phi'$ generated by the family of semi-norms $( \eta_{B} )$, where for each $B \subseteq \Phi$ bounded we have $\eta_{B}(f)=\sup \{ \abs{\inner{f}{\phi}}: \phi \in B \}$ for all $f \in \Phi'$.  Recall that $\Phi$ is called \emph{semi-reflexive} if the canonical (algebraic) embedding of $\Phi$ into $\Phi''$ is onto, and is called \emph{reflexive} if the canonical embedding is indeed an isomorphism (of topological vector spaces).


Let us recall that a (Hausdorff) locally convex space $(\Phi,\mathcal{T})$ is called \emph{nuclear} if its topology $\mathcal{T}$ is generated by a family $\Pi$ of Hilbertian semi-norms such that for each $p \in \Pi$ there exists $q \in \Pi$, satisfying $p \leq q$ and the canonical inclusion $i_{p,q}: \Phi_{q} \rightarrow \Phi_{p}$ is Hilbert-Schmidt. Other equivalent definitions of nuclear spaces can be found in \cite{Pietsch, Treves}. Recall that any quasi-complete bornological nuclear space is barreled and semi-reflexive, therefore is reflexive (see Theorem IV.5.6 in \cite{Schaefer}, p.145).

The following are all examples of complete, ultrabornological (hence barrelled) nuclear spaces: 
the spaces of functions $\mathscr{E}_{K} \defeq \mathcal{C}^{\infty}(K)$ ($K$: compact subset of $\R^{d}$) and $\mathscr{E}\defeq \mathcal{C}^{\infty}(\R^{d})$, the rapidly decreasing functions $\mathscr{S}(\R^{d})$, and the space of test functions $\mathscr{D}(U) \defeq \mathcal{C}_{c}^{\infty}(U)$ ($U$: open subset of $\R^{d}$), as well are the spaces of distributions $\mathscr{E}'_{K}$, $\mathscr{E}'$, $\mathscr{S}'(\R^{d})$, and $\mathscr{D}'(U)$. Other examples are the space of harmonic functions $\mathcal{H}(U)$ ($U$: open subset of $\R^{d}$), the space of polynomials $\mathcal{P}_{n}$ in $n$-variables and the space of real-valued sequences $\R^{\N}$ (with direct product topology). For references see \cite{Pietsch, Schaefer, Treves}.

\subsection{Cylindrical and stochastic processes}

Throughout this work we assume that $\ProbSpace$ is a complete probability space and consider a filtration $( \mathcal{F}_{t} : t \geq 0)$ on $\ProbSpace$ that satisfies the \emph{usual conditions}, i.e. it is right continuous and $\mathcal{F}_{0}$ contains all subsets of sets of $\mathcal{F}$ of $\Prob$-measure zero. We denote by $L^{0} \ProbSpace$ the space of equivalence classes of real-valued random variables defined on $\ProbSpace$. We always consider the space $L^{0} \ProbSpace$ equipped with the topology of convergence in probability and in this case it is a complete, metrizable, topological vector space. 
We denote by $\mathcal{P}_{\infty}$ the predictable $\sigma$-algebra on $[0, \infty) \times \Omega$ and for any $T>0$, we denote by $\mathcal{P}_{T}$ the restriction of $\mathcal{P}_{\infty}$ to $[0,T] \times \Omega$.  

Let $\Phi$ be a locally convex space. A \emph{cylindrical random variable}\index{cylindrical random variable} in $\Phi'$ is a linear map $X: \Phi \rightarrow L^{0} \ProbSpace$ (see \cite{FonsecaMora:Existence}). If $X$ is a cylindrical random variable in $\Phi'$, we say that $X$ is \emph{$n$-integrable} ($n \in \N$)  if $ \Exp \left( \abs{X(\phi)}^{n} \right)< \infty$, $\forall \, \phi \in \Phi$, and has \emph{mean-zero} if $ \Exp \left( X(\phi) \right)=0$, $\forall \phi \in \Phi$. 

Let $X$ be a $\Phi'$-valued random variable, i.e. $X:\Omega \rightarrow \Phi'$ is a $\mathscr{F}/\mathcal{B}(\Phi')$-measurable map. For each $\phi \in \Phi$ we denote by $\inner{X}{\phi}$ the real-valued random variable defined by $\inner{X}{\phi}(\omega) \defeq \inner{X(\omega)}{\phi}$, for all $\omega \in \Omega$. The linear mapping $\phi \mapsto \inner{X}{\phi}$ is called the \emph{cylindrical random variable induced/defined by} $X$. We will say that a $\Phi'$-valued random variable $X$ is (weakly) \emph{$n$-integrable} if the cylindrical random variable induced by $X$ is \emph{$n$-integrable}. 
 
Let $J=\R_{+} \defeq [0,\infty)$ or $J=[0,T]$ for  $T>0$. We say that $X=( X_{t}: t \in J)$ is a \emph{cylindrical process} in $\Phi'$ if $X_{t}$ is a cylindrical random variable for each $t \in J$. Clearly, any $\Phi'$-valued stochastic processes $X=( X_{t}: t \in J)$ \emph{induces/defines} a cylindrical process under the prescription: $\inner{X}{\phi}=( \inner{X_{t}}{\phi}: t \in J)$, for each $\phi \in \Phi$. 

If $X$ is a cylindrical random variable in $\Phi'$, a $\Phi'$-valued random variable $Y$ is called a \emph{version} of $X$ if for every $\phi \in \Phi$, $X(\phi)=\inner{Y}{\phi}$ $\Prob$-a.e. A $\Phi'$-valued process $Y=(Y_{t}:t \in J)$ is said to be a $\Phi'$-valued \emph{version} of the cylindrical process $X=(X_{t}: t \in J)$ on $\Phi'$ if for each $t \in J$, $Y_{t}$ is a $\Phi'$-valued version of $X_{t}$.  

For a $\Phi'$-valued process $X=( X_{t}: t \in J)$ terms like continuous, c\`{a}dl\`{a}g, purely discontinuous, adapted, predictable, etc. have the usual (obvious) meaning. 

A $\Phi'$-valued random variable $X$ is called \emph{regular} if there exists a weaker countably Hilbertian topology $\theta$ on $\Phi$ (see Section 2 in \cite{FonsecaMora:Existence}) such that $\Prob( \omega: X(\omega) \in (\widehat{\Phi}_{\theta})')=1$; here $\Phi_{\theta}$ denotes the space $(\Phi,\theta)$ and $\widehat{\Phi}_{\theta}$ denotes its completion. If $\Phi$ is barrelled, the property of being regular is  equivalent to the property that the law of $Y$ is a Radon measure on $\Phi'$ (see Theorem 2.10 in \cite{FonsecaMora:Existence}). A $\Phi'$-valued process $Y=(Y_{t}:t \in J)$ is said to be \emph{regular} if $Y_{t}$ is a regular random variable for each $t \in J$. 

A cylindrical process $Y=(Y_{t}:t \in J)$ in $\Phi'$ is a cylindrical martingale (respectively a cylindrical semimartingale) if $Y(\phi)=(Y_{t}(\phi): t \in J)$ is a real-valued martingale (respectively semimartingale) for each $\phi \in \Phi$. A $\Phi'$-valued process is a martingale (respectively a semimartingale) if the induced cylindrical process is a cylindrical martingale (respectively a cylindrical semimartingale)  in  $\Phi'$.

\subsection{Differentiation in the dual of a nuclear space}

Assume that $\Psi$ is a complete  bornological nuclear space. In this section we briefly recall the theory of differentiation for mappings defined on the strong dual $\Psi'$. We were not able to find an specific reference of this theory under our assumptions for $\Psi$, so we will first recall some key properties of $\Psi$ and then relate them with the general theory of differentiation in locally convex spaces (as for example in \cite{Yamamuro:1974}).  

Our assumptions on $\Psi$ show that $\Psi$ is a Montel space (see \cite{Schaefer}), so $\Psi$ is reflexive and its dual $\Psi'$ is barrelled (thus is equipped with its Mackey topology). Furthermore, being a complete nuclear space we know $\Psi$ is a projective system of Hilbert spaces, hence by Theorem IV.4.4 in \cite{Schaefer} we have $\Psi'$ is a inductive limit of Hilbert spaces. To be precise, if $\Pi$ is a collection of continuous Hilbertian seminorms on $\Psi$ generating its topology, then $\Psi$ is the projective limit $\displaystyle \lim_{\longleftarrow} (\Psi_{p}, i_{p,q})$ and the strong dual $\Psi'$ is the inductive limit $\displaystyle \lim_{\longrightarrow} (\Psi'_{p}, i'_{p,q})$.

We follow the definition of differentiation in \cite{Yamamuro:1974}. Let $F: D \subseteq \Psi' \rightarrow \Gamma$, where $D$ is open and $\Gamma$ is a locally convex space. We say that $F$ is \emph{Fr\'{e}chet differentiable}  at $\varphi \in D$ if there exists $\nu \in \mathcal{L}(\Psi',\Gamma)$ such that for every bounded subset $K$ of $\Psi'$, uniformly with respect to every $g \in K$, we have 
$$ \lim_{\epsilon \searrow 0} \epsilon^{-1} \left( F(\varphi +\epsilon g) -F(\varphi)-\epsilon \nu(g) \right) = 0. $$
We denote $\nu$ by $F'(\varphi)$. 

Since $\Psi' = \displaystyle \lim_{\longrightarrow} (\Psi'_{p}, i'_{p,q})$, by the result (1.6.1) in \cite{Yamamuro:1974}, $F$ is Fr\'{e}chet differentiable at $\varphi \in D$ if and only if the restriction $\restr{F}{\Psi'_{p}}$ of $F$ to $\Psi'_{p}$ is differentiable for each $p \in \Pi$ such that $\varphi \in \Psi'_{p}$. Furthermore, $\left( \restr{F}{\Psi'_{p}} \right)'=\restr{F'}{\Psi'_{p}}$. If $\Gamma$ is a Banach space, then $\left( \restr{F}{\Psi'_{p}} \right)'$ coincides with the usual definition of Fr\'{e}chet derivative between normed spaces. 

We say that $F$ is \emph{twice Fr\'{e}chet differentiable} at $\varphi \in D$ if $F$ is Fr\'{e}chet differentiable on $D$ and the mapping $F': D \rightarrow \mathcal{L}(\Psi', \Gamma)$ is Fr\'{e}chet differentiable at $\varphi$. We use the notation $F''(\varphi)$ for the second derivative of $F$ at $\varphi$. By definition $F''(\varphi) \in \mathcal{L}(\Psi', \mathcal{L}(\Psi',\Gamma))$, but for convenience we identify $F''(\varphi)$ with an element in $\mbox{Bil}(\Psi',\Psi';\Gamma)$ (the space of continuous bilinear forms on $\Psi' \times \Psi'$ with values in $\Gamma'$).

\section{Continuous and purely discontinuous part of a cylindrical martingale-valued measure}\label{secContAndDisconCMVM}

\begin{assumption} From now on $\Phi$ denotes a locally convex space and $\Psi$ denotes a complete bornological nuclear space. 
\end{assumption}

\begin{definition}\label{defiCMVM}
Let $U$ be a topological space and consider a ring $\mathcal{R}\subseteq \mathcal{B}(U)$ that generates $\mathcal{B}(U)$.  A \emph{cylindrical martingale-valued measure} on $\R_{+} \times \mathcal{R}$ is a collection $M=(M(t,A): t \geq 0, A \in \mathcal{R})$ of cylindrical random variables in $\Phi'$ such that:
\begin{enumerate}
\item $\forall \, A \in \mathcal{R}$, $M(0,A)(\phi)= 0$ $\Prob$-a.s., $\forall \phi \in \Phi$.
\item $\forall t \geq 0$, $M(t,\emptyset)(\phi)= 0$ $\Prob$-a.s. $\forall \phi \in \Phi$ and if $A, B \in \mathcal{R}$ are disjoint then 
$$M(t,A \cup B)(\phi)= M(t,A)(\phi) + M(t,B)(\phi) \, \Prob \mbox{-a.s.}, \quad \forall \phi \in \Phi.$$
\item $\forall \, A \in \mathcal{R}$, $(M(t,A): t \geq 0)$ is a cylindrical mean-zero square integrable c\`{a}dl\`{a}g martingale. 
\end{enumerate}
We say $M$ is \emph{orthogonal} if
\begin{enumerate}\setcounter{enumi}{3}
\item For disjoint $A, B \in \mathcal{R}$, $\Exp \left( M(t,A)(\phi) M(s,B)(\varphi) \right)=0$, for each $t,s \geq 0$, $\phi, \varphi \in \Phi$. 
\end{enumerate}
We say $M$ has \emph{independent increments} if
\begin{enumerate}\setcounter{enumi}{4}
\item For $0\leq s < t$, $(M(t,A)- M(s,A))(\phi)$ is independent of $\mathcal{F}_{s}$, for all $A \in \mathcal{R}$, $\phi \in \Phi$.
\end{enumerate}
Further, we say $M$ is \emph{nuclear} if 
\begin{enumerate}\setcounter{enumi}{5}
\item  For each $A \in \mathcal{R}$, $0 \leq s < t$, $\phi \in \Phi$,  
\begin{equation}\label{eqDefiNuclearCMVM}
\Exp \left( \abs{ (M(t,A)- M(s,A))(\phi)}^{2} \right) = \int_{s}^{t} \int_{A} q_{r,u}(\phi)^{2} \mu(du) dr ,
\end{equation} 
where $\mu$ is a $\sigma$-finite measure on $(U, \mathcal{B}(U))$ satisfying $\mu(A)< \infty$ $\forall \, A \in \mathcal{R}$, and $\{q_{r,u}: r \in \R_{+}, \, u \in U \}$ is a family of  continuous Hilbertian semi-norms on $\Phi$ such that for each $\phi$, $\varphi$ in $\Phi$, the map $(r,u) \mapsto q_{r,u}(\phi,\varphi)$ is $\mathcal{B}(\R_{+}) \otimes \mathcal{B}(U)/ \mathcal{B}(\R_{+})$-measurable and  bounded on $[0,T] \times U$ for all $T>0$. Here, $q_{r,u}(\cdot,\cdot)$ denotes the positive, symmetric, bilinear form associated to the Hilbertian semi-norm $q_{r,u}$.  
\end{enumerate}
\end{definition}

In \cite{FonsecaMora:SPDECMVM}, the author constructed a theory of stochastic integration with respect to the above class of (orthogonal, with independent increments, nuclear) cylindrical martingale-valued measures in a two step construction. First, a real-valued stochastic integral, known as the weak stochastic integral, is constructed via an It\^{o} isometry. Later, by using a procedure of regularization a vector-valued stochastic integral, known as the strong stochastic integral, is constructed using the weak stochastic integral as a building block. 

The basic class of (strong) stochastic integrands is the following: 
Let $\Lambda^{2}_{s}(\Psi, M, T)$ denote the collection of families $R=\{R(r,\omega,u): r \in [0,T], \omega \in \Omega, u \in U \}$ of operator-valued maps satisfying the following conditions:
\begin{enumerate}
\item $R(r,\omega,u) \in \mathcal{L}(\Phi'_{q_{r,u}},\Psi')$, for all $r \in [0, T]$, $\omega \in \Omega$, $u \in U$, 
\item $R$ is \emph{$q_{r,u}$-predictable}, i.e. for each $\phi \in \Phi$, $\psi \in \Psi$, the mapping $[0,T] \times \Omega \times U \rightarrow \R_{+}$ given by $(r,\omega,u) \mapsto q_{r,u}(R(r,\omega,u)' \psi, \phi)$ is $\mathcal{P}_{T} \otimes \mathcal{B}(U)$-measurable.
\item 
\begin{equation} \label{finiteSecondMomentIntegrandsStrongIntg}
\Exp \int_{0}^{T} \int_{U} q_{r,u}(R(r,u)'\psi)^{2} \mu(du) dr < \infty, \quad \forall \, \psi \in \Psi.
\end{equation} 
\end{enumerate}
If $R \in \Lambda^{2}_{s}(\Psi, M, T)$, Theorem 5.11 in  \cite{FonsecaMora:SPDECMVM} shows the existence of a $\Psi'$-valued mean-zero, square integrable, c\`{a}dl\`{a}g martingale 
$ \int_{0}^{t}\int_{U} \, R(r,u) \, M(dr,du)$, $0 \leq t \leq T$,
called the \emph{stochastic integral} of $R$ with respect to $M$. For every $\psi \in \Psi$, $\Prob$-a.e. $\forall t \in [0,T]$, we have the following \emph{weak-strong compatibility}:
\begin{equation}\label{eqWeakStrongCompLevyMartValMeas}
 \inner{\int^{t}_{0} \int_{U} R (r,f) M (dr, df)}{\psi}=\int^{t}_{0} \int_{U} R (r,f)'\psi M (dr, df),
\end{equation}
where the real-valued (weak) stochastic integral on the right-hand side of the above equality is defined by Theorem 4.7 in \cite{FonsecaMora:SPDECMVM}.  Below, familiarity with the theory of stochastic integration in the dual of a nuclear space will be assumed, the reader is referred to \cite{FonsecaMora:SPDECMVM} for the details.


As the next result shows, we can decompose $M$ as the sum of a cylindrical  martingale-valued measure $M^{c}$ with continuous paths and  a cylindrical  martingale-valued measure $M^d$ with purely discontinuous paths. Recall that any real-valued  martingale $m$ can be written uniquely as $m = m^{c}+ m^{d}$ where $m^{c}$ is a continuous martingale with $m^{c}_{0}=0$  and $m^{d}$ is a  purely discontinuous martingale (see e.g. Theorem 10.3.4 in \cite{CohenElliott}). 

\begin{lemma}
Let $M$ be a cylindrical martingale-valued measure on $\R_{+} \times \mathcal{R}$. There exists two  cylindrical martingale-valued measures $M^{c}$ and $M^{d}$ on $\R_{+} \times \mathcal{R}$ with the following properties:
\begin{enumerate}
\item $M(t,A)=M^{c}(t,A)+M^{d}(t,A)$ $\forall t\geq 0$, $A \in \mathcal{R}$.  
\item For every $A \in \mathcal{R}$ and $\phi \in \Phi$, $(M^{c}(t,A)(\phi): t \geq 0)$ has continuous paths and $(M^{d}(t,A)(\phi): t \geq 0)$ has purely discontinuous paths. 
\item For every $A \in \mathcal{R}$ and $\phi \in \Phi$, the square integrable martingales $(M^{c}(t,A)(\phi): t \geq 0)$ and $(M^{d}(t,A)(\phi): t \geq 0)$ are orthogonal.   
\end{enumerate}
\end{lemma}
\begin{proof}
For every $t\geq 0$, $A \in \mathcal{R}$ and $\phi \in \Phi$, define  
$$
M^{c}(t,A)(\phi)=[M(t,A)(\phi)]^{c}, \quad M^{d}(t,A)(\phi)=[M(t,A)(\phi)]^{d}.
$$
We must prove that the families $M^{c}=(M^{c}(t,A): t \geq 0, A \in \mathcal{R})$ and $M^{d}=(M^{d}(t,A): t \geq 0, A \in \mathcal{R})$ are cylindrical martingale-valued measures.

First, since $M(0,A)(\phi)= 0$ and by construction $M^{c}(0,A)(\phi)= 0$, we must have $M^{d}(0,A)(\phi)= 0$. Second, since $M(t,\emptyset)(\phi)= 0$, we have $M^{d}(t,\emptyset)(\phi)= -M^{c}(t,\emptyset)(\phi)$, but since $M^{d}(t,\emptyset)(\phi)$ is purely discontinuous, 
we have 
$$0= \Exp \left[ M^{d}(t,\emptyset)(\phi) M^{c}(t,\emptyset)(\phi)\right]  = - \Exp \left[ M^{c}(t,\emptyset)(\phi)M^{c}(t,\emptyset)(\phi)\right], $$
which shows  $M^{c}(t,\emptyset)(\phi)=0$. This in turn implies $M^{d}(t,\emptyset)(\phi)=0$. If $A, B \in \mathcal{R}$ are disjoint then 
\begin{eqnarray*}
M(t,A \cup B)(\phi)  & = & M(t,A)(\phi) + M(t,B)(\phi) \\
& = &  \left[ M^{c}(t,A)(\phi) + M^{c}(t,B)(\phi) \right] + \left[ M^{d}(t,A)(\phi) + M^{d}(t,B)(\phi) \right].
\end{eqnarray*}
The first two terms in the right-hand side of the last line of the above equality are continuous, and the other two are purely discontinuous, the uniqueness of the representation gives us
$$ M^{c}(t,A \cup B)(\phi) = M^{c}(t,A)(\phi) + M^{c}(t,B)(\phi), $$
and likewise for $M^{d}$. Third, we must prove that $\phi \mapsto M^{c}(t,A)(\phi)$ and $\phi \mapsto M^{d}(t,A)(\phi)$ are linear mappings. 
As before, this can be proved from the equality
\begin{flalign*}
& M(t, A)(\alpha \phi_1 + \phi_2)  \\
& = \left[ \alpha M^c(t, A)(\phi_1) + M^c(t, A)(\phi_2) \right] + \left[ \alpha M^d(t, A)(\phi_1) + M^d(t, A)(\phi_2) \right]. 
\end{flalign*}
and arguments that use the uniqueness of the decomposition. Therefore, $(M^{c}(t,A): t \geq 0)$ and $(M^{d}(t,A): t \geq 0)$ are cylindrical mean-zero square integrable martingales.
\end{proof}

\begin{remark} For each $A, B \in \mathcal{R}$, $t,s \geq 0$ and $\phi, \varphi \in \Phi$, from the fact that $M^{d}(\cdot,A)(\phi)$ and $M^{d}(\cdot,B)(\varphi)$ are purely discontinuous,  we have 
\begin{multline}\label{eqSquareMomenMcAndMd}
\Exp \left( M(t,A)(\phi) M(s,B)(\varphi) \right) \\
 =  \Exp \left( M^{c}(t,A)(\phi) M^{c}(s,B)(\varphi) \right)+ \Exp \left( M^{d}(t,A)(\phi) M^{d}(s,B)(\varphi) \right).
\end{multline}
From the above we conclude that if $M^{c}$ and $M^{d}$ are orthogonal, so  is $M$. The converse is not clear to be true in general.  
\end{remark}

\begin{assumption}\label{assumOnMcAndMd}
From  now on we assume the following:
\begin{enumerate}
\item $M$ is a cylindrical martingale-valued measure on $\R_{+} \times \mathcal{R}$. 
\item Both $M^{c}$ and $M^{d}$ are orthogonal and have independent increments. 
\item Both $M^{c}$ and $M^{d}$ are nuclear  satisfying \eqref{eqDefiNuclearCMVM} with respect to the same measure $\mu$ and for families of seminorms  $\{q^{c}_{r,u}: r \in \R_{+}, \, u \in U \}$ for $M^{c}$ and  $\{q^{d}_{r,u}: r \in \R_{+}, \, u \in U \}$ for $M^{d}$.  
\end{enumerate}
\end{assumption}

From the above assumption we have that $M$ is orthogonal and has independent increments. Moreover, taking $A=B$, $s=t$ and $\phi=\varphi$ in \eqref{eqSquareMomenMcAndMd}, we conclude from \eqref{eqDefiNuclearCMVM} for $M^{c}$ and $M^{d}$ that $M$ is nuclear with corresponding family of seminorms  
\begin{equation}\label{eqDecomSeminormsMcAndMd}
q_{r,u}(\phi)^{2} = q^{c}_{r,u}(\phi)^{2}+ q^{d}_{r,u}(\phi)^{2}, \quad \forall \,  r \in \R_{+}, \, u \in U .
\end{equation}

\begin{proposition}\label{propStochIntegralSumContAndDiscParts}
Assume $R \in \Lambda^{2}_{s}(\Psi, M, T)$ and denote by $j_{c}:\Phi'_{q^{c}_{r,u}} \rightarrow \Phi'_{q_{r,u}}$ and $j_{d}: \Phi'_{q^{d}_{r,u}} \rightarrow \Phi'_{q_{r,u}}$ the canonical inclusions. Then, $\{R(r,\omega, u)  j_{c} \} \in \Lambda^{2}_{s}(\Psi, M^{c}, T)$ and $\{R(r,\omega, u)  j_{d} \} \in \Lambda^{2}_{s}(\Psi, M^{d}, T)$. Moreover, for all $0 \leq t \leq T$, 
\begin{equation}\label{eqDecompStochastInteContinuAndPurelyDiscon}
\int_{0}^{t}\int_{U} \, R(r,u) \, M(dr,du)= \int_{0}^{t}\int_{U} \, R(r,u)j_{c} \, M^{c}(dr,du)+ \int_{0}^{t}\int_{U} \, R(r,u)j_{d} \, M^{d}(dr,du). 
\end{equation}
\end{proposition}
\begin{proof} Observe first that by \eqref{eqDecomSeminormsMcAndMd} we have $\norm{j_{c}}_{\mathcal{L}(\Phi'_{q^{c}_{r,u}}, \Phi'_{q_{r,u}})} \leq 1$, $\norm{j_{d}}_{\mathcal{L}(\Phi'_{q^{d}_{r,u}},  \Phi'_{q_{r,u}})} \leq 1$. From this one can deduce that $\{R(r,\omega, u)  j_{c} \} \in \Lambda^{2}_{s}(\Psi, M^{c}, T)$ and $\{R(r,\omega, u)  j_{d} \} \in \Lambda^{2}_{s}(\Psi, M^{d}, T)$. To prove \eqref{eqDecompStochastInteContinuAndPurelyDiscon}, by Theorem 5.11 in \cite{FonsecaMora:SPDECMVM} it is enough to show that for every $\psi \in \Psi$ we have 
\begin{multline*}
\int_{0}^{t}\int_{U} \, R(r,u)' \psi \, M(dr,du) \\
= \int_{0}^{t}\int_{U} \, j'_{c} R(r,u)' \psi \, M^{c}(dr,du)+ \int_{0}^{t}\int_{U} \,  j'_{d} R(r,u)' \psi \, M^{d}(dr,du). 
\end{multline*}
where the above are all (real-valued) stochastic integrals (see Theorem 4.7 in \cite{FonsecaMora:SPDECMVM}). 
The proof of the above identity can be carried out by standard arguments checking that it holds for simple integrands and then extending the identity by an approximation procedure (see the proof of Proposition 4.11 in \cite{FonsecaMora:SPDECMVM}). 
\end{proof}

\section{Generalized It\^{o} processes}\label{sectionGeneItoProcesses}

In this section  our main objective is to give a proper definition for a $\Psi'$-valued process of the form 
$$ dX_{t}= B(t) dt + \int_{U} R(t,u) M(dt,du). $$
It is natural to call $B$ as the \emph{drift coefficient} and to $R$ as the \emph{diffusion coefficient}.

\begin{assumption}\label{assumItoProcess} From now on we assume the following:
\begin{enumerate}
\item $\xi$ is a $\Psi'$-valued $\mathcal{F}_{0}$-measurable regular random variable. We assume further that $\xi$ is integrable, that is, $\Exp \left( \abs{\inner{\xi}{\psi}}\right) < \infty$,  $\forall \, \psi \in \Psi$. 
\item $B: \R_{+} \times \Omega \rightarrow \Psi'$ is such that for each $\psi \in \Psi$ the mapping $(r,\omega) \mapsto \inner{B(r,\omega)}{\psi}$ is progressively measurable. 
\item $R=\{R(r,\omega,u):  r \in \R_{+}, \omega \in \Omega, u \in U  \} $ is a family satisfying $R(r,\omega,u) \in \mathcal{L}(\Phi'_{q_{r,u}},\Psi')$ and is $q_{r,u}$-predictable, i.e. the mapping $(r,\omega,u) \mapsto q_{r,u}(R(r,\omega,u)' \psi, \phi)$ is  $\mathcal{P} \otimes \mathcal{B}(U)$-measurable for every $\psi \in \Psi$ and $\phi \in \Phi$.
\item There exists a sequence of bounded stopping times $(\sigma_{n}: n \in \N)$ increasing to $\infty$ such that 
\begin{equation}\label{eqLocalIntegCoeffiB}
 \Exp \int_{0}^{\sigma_{n}} \abs{\inner{B(r)}{\psi}} dr < \infty, \quad \forall \, \psi \in \Psi. 
\end{equation}
\begin{equation} 
\Exp \int_{0}^{\sigma_{n}} \int_{U} q_{r,u}(R(r,u)'\psi)^{2} \mu(du) \lambda(dr) < \infty, \quad \forall \, \psi \in \Psi.
\end{equation}
\end{enumerate}
\end{assumption}

With the assumptions above, one can define (see \cite{FonsecaMora:SPDELevy}, Section 3.2) the integral for the drift coefficient as a $\Psi'$-valued continuous process satisfying $\Prob$-a.e.
\begin{equation}\label{eqWeakStrongCompaLebesIntegral}
\inner{\int_{0}^{t} B(r)dr}{\psi} = \int_{0}^{t} \inner{B(r)}{\psi} dr, \quad \forall \, \psi \in \Psi .
\end{equation}
Likewise, for the diffusion coefficient we have that by Theorem 5.26 in \cite{FonsecaMora:SPDECMVM} there exists the stochastic integral process 
$$\left(  \int_{0}^{t} \int_{U} R(r,u) M(dr,du) : t \geq 0 \right)$$
 which is a $\Psi'$-valued locally square integrable c\`{a}dl\`{a}g martingale.

\begin{definition}\label{defiItoProcess}
With $B$ and $R$ as described above, we define the \emph{generalized It\^{o} process}:
$$X_{t}= \xi+\int_{0}^{t} B(r)dr+ \int_{0}^{t} \int_{U} R(r,u) M(dr,du), \quad t \geq 0. $$
\end{definition}

The process $(X_{t}: t \geq 0)$ is a $\Psi'$-valued adapted c\`{a}dl\`{a}g  semimartingale. But we can show that locally it its a c\`{a}dl\`{a}g semimartingale in a Hilbert space. 

\begin{theorem}\label{theoLocalHilbertSemima}
There exists an increasing sequence of continuous Hilbertian seminorms $(\varrho_{n}: n \in \N)$ on $\Psi$ with the property that for all $n \in \N$, $X^{\sigma_{n}}$ is a $\Psi'_{\varrho_{n}}$-valued c\`adl\`ag semimartingale. 
\end{theorem}
\begin{proof}
Let $(T_{n}: n \in \N)$ be an increasing sequence of positive real numbers such that $\forall \, n \in \N$, $\sigma_{n} \leq T_{n}$ $\Prob$-a.e. and $T_{n} \nearrow \infty$. By Lemma 6.11 and  Theorem 5.11 in \cite{FonsecaMora:SPDECMVM} one can show the existence of an increasing sequence $(\varrho_{n}: n \in \N)$  of continuous Hilbertian seminorms on $\Psi$ with the following properties: 
\begin{enumerate}
\item For every $n \in \N$, on $[0,T_{n}]$ we have $B(r)$ has a $\Psi'_{\varrho_{n}}$-valued version $\tilde{B}_{n}(r)$ and 
$$ \Exp \int_{0}^{T_{n}} \mathbbm{1}_{[0,\sigma_{n}]}(r) \varrho_{n}'(\tilde{B}(r)) dr < \infty.$$
In particular,  $\int_{0}^{t} \, \mathbbm{1}_{[0,\sigma_{n}]}(r) \tilde{B}_{n}(r) dr$, $t \in [0,T_{n}]$, can be defined as a Bochner integral in $\Psi'_{\varrho_{n}}$ $\Prob$-a.e. Moreover, by a standard localization argument
 \begin{equation}\label{eqLocalDriftIntegral}
 \int_{0}^{t \wedge \sigma_{n} } B(r) dr = \int_{0}^{t} \, \mathbbm{1}_{[0,\sigma_{n}]}(r) \tilde{B}_{n}(r) dr, \quad \forall t \in [0,T_{n}]. 
 \end{equation}

\item  For every $n \in \N$ there exists a family 
$\tilde{R}_{n}=\{\tilde{R}_{n}(r,\omega,u)\}$ of linear operators $\tilde{R}_{n}(r,\omega,u) \in \mathcal{L}_{2}(\Phi'_{q_{r,u}},\Psi'_{\varrho_{n}})$, $r \in [0, T_{n}]$, $\omega \in \Omega$, $u \in U$, which are $q_{r,u}$-predictable, and for which 
\begin{equation} 
\Exp \int_{0}^{T_{n}} \int_{U} \mathbbm{1}_{[0,\sigma_{n}]}(r) \norm{\tilde{R}_{n}(r,u)}^{2}_{\mathcal{L}_{2}(\Phi'_{q_{r,u}},\Psi'_{\varrho_{n}})} \mu(du) dr < \infty,
\end{equation} 
and the stochastic integral 
$ \int_{0}^{t} \int_{U} \mathbbm{1}_{[0,\sigma_{n}]}(r)  \tilde{R}_{n}(r,u) M(dr,du)$ is a $\Phi'_{\varrho_{n}}$-valued mean-zero (strongly) square integrable martingale with c\`{a}dl\`{a}g paths satisfying 
\begin{multline}\label{eqItoIsometryStrongInte}
\Exp \left( \varrho_{n}'\left( \int_{0}^{t} \int_{U} \mathbbm{1}_{[0,\sigma_{n}]}(r)  \tilde{R}_{n}(r,u) M(dr,du) \right)^{2} \right) \\
= \Exp \int_{0}^{t} \int_{U} \mathbbm{1}_{[0,\sigma_{n}]}(r)  \norm{\tilde{R}_{n}(r,u)}^{2}_{\mathcal{L}_{2}(\Phi'_{q_{r,u}},\Psi'_{\varrho_{n}})} \mu(du) dr. 
\end{multline}
Moreover,  $ \mathbbm{1}_{[0,\sigma_{n}]}(r) R(r,\omega,u)=i'_{p}  \mathbbm{1}_{[0,\sigma_{n}]}(r) \tilde{R}_{n}(r,\omega,u)$ for $\mbox{Leb} \otimes \Prob \otimes \mu$-a.e. $(r,\omega,u)$ and 
\begin{multline}\label{eqLocalDiffusionIntegral}
 \int_{0}^{t \wedge \sigma_{n} } \int_{U} R(r,\omega,u) M(dr,du) \\
 = \int_{0}^{t} \int_{U}  \mathbbm{1}_{[0,\sigma_{n}]}(r) \tilde{R}_{n}(r,\omega,u) M(dr,du), \quad \forall t \in [0,T_{n}]. 
\end{multline}
\end{enumerate}
Finally, since $\xi$ is (weakly) integrable, by  a modification in the arguments of Lemma 6.11 in \cite{FonsecaMora:SPDECMVM} there exists a continuous Hilbertian seminorm $\rho$ on $\Psi$ such that $\xi$ has a $\Psi'_{\rho}$-valued version $\tilde{\xi}$ with $\Exp \left(\rho'(\tilde{\xi})\right)<\infty$. One can assume without loss of generality that $\rho=\varrho_{1}$. This way, if for every $n \in \N$ we define $\tilde{\xi}_{n}=i'_{\varrho_{1},\varrho_{n}} \tilde{\xi} $, then we have $\tilde{\xi}_{n}$ is a $\Psi'_{\varrho_{n}}$-valued $\mathcal{F}_{0}$-measurable (strongly) integrable random variable.   Then, by \eqref{eqLocalDriftIntegral} and \eqref{eqLocalDiffusionIntegral},  for every $n \in \N$ we have 
$$  \tilde{X}^{n}_{t} \defeq  X_{t \wedge \sigma_{n}} = \tilde{\xi}_{n} +\int_{0}^{t} \, \mathbbm{1}_{[0,\sigma_{n}]}(r) \tilde{B}_{n}(r) dr +  \int_{0}^{t} \int_{U}  \mathbbm{1}_{[0,\sigma_{n}]}(r) \tilde{R}_{n}(r,\omega,u) M(dr,du), $$
is a $\Psi'_{\varrho_{n}}$-valued c\`{a}dl\`{a}g semimartingale. 
\end{proof}

\begin{remark}\label{remarkQuadraticVariation}
Since $\tilde{X}^{n}$ is a $\Psi'_{\varrho_{n}}$-valued c\`adl\`ag semimartingale, the (optional) quadratic  variation $[\tilde{X}^{n}]_{t}$ of $\tilde{X}^{n}$ exists. We can therefore define the (optional) \emph{quadratic variation} $[X]$ of $X$ by localization as $[X]_{t \wedge \sigma_{n}}=[\tilde{X}^{n}]_{t}$ for $t \in [0,\sigma_{n}]$. 
\end{remark}

\section{It\^{o}'s formula for cylindrical martingale-valued measures}\label{sectItoFormulaCMVM}

As part of our formulation of the It\^{o}'s formula, we will need the following notion of trace with respect to the Hilbert spaces determined by the seminorms corresponding to the continuous part $M^{c}$ of $M$. 

\begin{definition}\label{defiTrace}
Let $A: \R_{+} \times \Omega \times U \rightarrow \mbox{Bil}(\Psi', \Psi'; H)$, where $H$ is a separable Hilbert space. Given $(s,\omega, u) \in \R_{+} \times \Omega \times U$, let $(e_{j}(s,u): j \in \N)$ be an orthonormal basis in $\Phi'_{q_{r,u}^{c}}$.  Define 
$$ \mbox{Tr}_{R(s,\omega, u)}(A(s,\omega,u))= \sum_{j=1}^{\infty} A(s,\omega,u) \left( R(s,\omega, u)  e_{j}(s,u), R(s,\omega, u)  e_{j}(s,u) \right), $$
provided the series is convergent in $H$. 
\end{definition}

\begin{remark}
By similar arguments to those used in the proof of Proposition 3.2 in \cite{Jakubowski:1995}, we can show that, by the properties of $R$, for all $(s,\omega, u) \in \R_{+} \times \Omega \times U$ the series in the definition of $\mbox{Tr}_{R(s,\omega, u)}(A(s,\omega,u))$ is convergent and that it does not depend on the choice of  orthonormal basis in $\Phi'_{q_{r,u}^{c}}$. Moreover, if for every $f,g \in \Psi'$ the mapping $(s,\omega, u) \mapsto A(s,\omega, u)(f,g)$ is $\mathcal{P} \otimes \mathcal{B}(U)/\mathcal{B}(H)$ measurable, then $(s,\omega, u) \mapsto  \mbox{Tr}_{R(s,\omega, u)}(A(s,\omega,u))$ is so. 
\end{remark}

The following is our main result.

\begin{theorem}[It\^{o}'s formula]\label{theoItoFormula}
Let $X$ be as in Definition \ref{defiItoProcess} and assume $\Gamma$ is a separable complete bornological nuclear space. Let $F:[0, \infty) \times \Psi' \rightarrow \Gamma'$ be a function of class $C^{1,2}$. We assume $F_{t}$ and $F_{xx}$ are uniformly continuous on bounded subsets of $[0,\infty)\times \Psi'$. Then, almost surely we have for every $t\geq 0$
$$
\begin{aligned}
& F(t, X_t) = F(0,\xi)+\int_{0}^{t} F_{t}(s, X_{s-}) d s + \int_{0}^{t}\!\!\int_U \, F_{x}(s, X_{s-})R(s, u) \, M(ds, du) \\
& + \int_{0}^{t} \, F_{x}(s, X_{s-}) B(s) \, ds  +\frac{1}{2} \int_{0}^{t} \!\!\int_U \textup{Tr}_{R(s, u)} \left( F_{x x}(s, X_{s-}) \right) \mu(du) ds \\
& +\sum_{0<s \leq t}\left[F\left(s,X_s\right)-F\left(s,X_{s-}\right)-F_x\left(s,X_{s-}\right)\left(\Delta X_{s} \right)\right].
\end{aligned}
$$
\end{theorem}

\begin{remark}\label{remaDefiTraceIntegral}
The integral $\int_{0}^{t} \!\!\int_U \textup{Tr}_{R(s, u)} \left( F_{x x}(s, X_{s-}) \right) \mu(du) dr$ is a defined as a $\Gamma'$-valued regular continuous  process satisfying $\Prob$-a.e.
\begin{multline*}
\inner{\int_{0}^{t} \!\!\int_U \textup{Tr}_{R(s, u)} \left( F_{x x}(s, X_{s-}) \right) \mu(du) ds}{\gamma} \\
= \int_{0}^{t} \!\!\int_U \textup{Tr}_{R(s, u)} \left( \inner{F_{x x}(s, X_{s-})}{\gamma} \right) \mu(du) ds, \quad \forall \gamma \in \Gamma.     
\end{multline*}
As part of the proof of Theorem \ref{theoItoFormula}, we will show  that the integral in the right-hand side is well-defined for every $\gamma \in \Gamma$ (observe that since the bilinear form  $\inner{F_{x x}(s, X_{s-})}{\gamma}$ is real-valued, the trace is as in Definition \ref{defiTrace}). Then the existence of the  $\Gamma'$-valued regular continuous process in the left-hand side of the above equality is obtained through a regularization theorem argument (see \cite{FonsecaMora:SPDELevy}, Section 3.2).   
\end{remark}

\begin{remark}\label{remaFiniteRandomSum} It is a consequence of Theorem \ref{theoLocalHilbertSemima} that for every $n \in \N$ and $\omega \in \Omega$, the mapping $s \mapsto X_{s}(\omega)$ is c\`{a}dl\`{a}g from $[0,\sigma_{n}(\omega)]$ into the Hilbert space $\Phi'_{\varrho_{n}}$. Because $\Phi'_{\varrho_{n}}$ is a complete separable metric space, then $\Delta X_{s}(\omega) \neq 0$ for a finite number of $s \in [0,\sigma_{n}(\omega)]$. Thus by our assumptions of continuity on $F$ and $F_{x}$, we conclude that 
$$ \sum_{0<s \leq t}\left[F\left(s,X_s\right)-F\left(s,X_{s-}\right)-F_x\left(s,X_{s-}\right)\left(\Delta X_{s} \right)\right],$$ 
is a finite random sum, hence is well-defined. 
\end{remark}

For the proof of Theorem \ref{theoItoFormula} (specially in Step 3 below), the reader will be assumed to be familiar with the theory of quadratic variation and of stochastic integration with respect to cylindrical martingale-valued measures in Hilbert spaces, see \cite{CambroneroEtAl:ItoFormula} for further details.

\begin{proof}[Proof of Theorem \ref{theoItoFormula}]
We first reduce the proof of the formula to the scalar case. Assume that  It\^{o}'s formula holds for $G:[0, \infty) \times \Psi' \rightarrow \R$. 

For $\gamma \in \Gamma$ consider the function $G:[0, \infty) \times \Psi' \rightarrow \R$ given by $G(t,x)=\inner{F(t,x)}{\gamma}$. Then $G$ is of class  $C^{1,2}$, $G_{t}$ and $G_{xx}$ are uniformly continuous on bounded subsets of $[0,\infty)\times \Psi'$. Applying It\^{o}'s formula to $G$, \eqref{eqWeakStrongCompLevyMartValMeas}, \eqref{eqWeakStrongCompaLebesIntegral}, and by Remarks \ref{remaDefiTraceIntegral} and  \ref{remaFiniteRandomSum}, we have $\Prob$-a.e. 
\begin{eqnarray*}
\inner{F(t, X_t)}{\gamma} 
& = &  \left\langle F(0,\xi)+\int_{0}^{t} F_{t}(s, X_{s-}) d s + \int_{0}^{t}\!\!\int_U \, F_{x}(s, X_{s-})R(s, u) \, M(ds, du) \right.\\
& {} & + \int_{0}^{t} \, F_{x}(s, X_{s-}) B(s) \, ds  +\frac{1}{2} \int_{0}^{t} \!\!\int_U \textup{Tr}_{R(s, u)} \left( F_{x x}(s, X_{s-}) \right) \mu(du) dr \\ 
& {} &  \left. +\sum_{0<s \leq t}\left[F\left(s,X_s\right)-F\left(s,X_{s-}\right)-F_x\left(s,X_{s-}\right)\left(\Delta X_{s} \right)\right] \, , \gamma \right\rangle 
\end{eqnarray*}
Since $\Gamma$ is separable, there exists a countable dense subset $D$ of $\Gamma$. If $M$ is the $\Q$-span of $D$, then $\Prob$-a.e. we have the above identity holds true $\forall \, \gamma \in M$. Then, by an application of the Hahn-Banach theorem we obtain the equality $\Prob$-a.e. in It\^{o}'s formula for the vector case. 

Our main objective then is to show that It\^{o}'s formula holds true for $F:[0, \infty) \times \Psi' \rightarrow \R$. For the reader's convenience we divide the  proof in three main steps. 

\textbf{Step 1.} We will show that $F$, $F_{t}$, $F_{x}$ and $F_{xx}$  are uniformly continuous and bounded on bounded subsets of $[0,\infty) \times \Psi'$.
 
Let $B$ a bounded subset of $[0,\infty) \times \Psi'$. Since $\Psi$ is barrelled, there exists $T>0$, a continuous Hilbertian seminorm $\rho$ on $\Psi$ and $R>0$ such that (see Theorem IV.5.2 in \cite{Schaefer})
$$ B \subseteq \tilde{B} \defeq [0,T] \times B_{\rho'} (R). $$ 
Let $F_{\rho}=\restr{F}{\Psi'_{\rho}}$. We have by the properties of the Fr\'echet derivative that
$$ (F_{\rho})_{x}= \restr{F_{x}}{\Psi'_{\rho}}, \quad (F_{\rho})_{xx} = \restr{F_{xx}}{\Psi'_{\rho}} . $$ 

Since $\tilde{B}$ is bounded on $\Psi'$, by our hypothesis we have $F_{xx}$ is uniformly continuous on $\tilde{B}$, hence bounded on $\tilde{B}$. Since $
\tilde{B} \subseteq [0,\infty) \times \Psi'_{\rho}$, we have $(F_{\rho})_{xx}$ is bounded on $\tilde{B}$. Recall that since $\Psi'_{\rho}$ is Hilbert, then $(F_{\rho})_{xx}$ coincides with the usual definition of Fr\'echet derivative for Hilbert spaces. Let 
$$ K= \sup \left\{ \norm{(F_{\rho})_{xx}(s,x)}_{\mbox{Bil}(\Psi'_{\rho}, \Psi'_{\rho})}: (s,x) \in \tilde{B} \right\}. $$
By the mean-value theorem, on $\tilde{B}$ we have 
$$ \norm{ (F_{\rho})_{x}(s,x)- (F_{\rho})_{x}(t,y) }_{\mathcal{L}(\Psi'_{\rho},\R)} \leq K \left( \abs{t-s}+\rho'(x-y) \right). $$
Thus $(F_{\rho})_{x}$ is uniformly continuous on $B$, hence bounded on $B$. Therefore $F_{x}$ is uniformly continuous and bounded on $B$. An analogous application of the mean-value theorem shows that $F$ is uniformly continuous and bounded on $B$. This proves Step 1. 

\textbf{Step 2.} We will show that all the processes in  It\^{o}'s formula are well-defined. Moreover, we will prove that It\^{o}'s formula holds true as long as it is true under the following assumption:

\textbf{Assumption A.} \emph{There exists a continuous Hilbertian seminorm $\eta$ on $\Psi$ such that $\xi(\omega) \in \Psi'_{\eta}$, $B(r,\omega) \in \Psi'_{\eta}$, $R(r,\omega,u) \in \mathcal{L}_{2}(\Phi'_{q_{r,u}},\Psi'_{\eta})$ $\forall u \in U$, $t \in [0,T]$ $\Prob$-a.e., and there exists $K>0$ such that $\Prob$-a.e.}
$$ \eta'(X_{t}) \leq K, \quad \int_{0}^{T} \, \eta'(B(r)) dr \leq K, \quad \int_{0}^{T} \int_{U} \norm{R(r,u)}^{2}_{\mathcal{L}_{2}(\Phi'_{q_{r,u}},\Psi'_{\eta})} \mu(du) dr \leq K. $$

To prove Step 2, for $n \in \N$, $k>0$, let 
\begin{multline*}
\tau_{n,k} = T \wedge \sigma_{n} \wedge 
\inf \left\{ t: \varrho_{n}'(X_{t}) \wedge \int_{0}^{T} \, \varrho_{n}'(B(r)) dr \right. \\
\left.  \wedge \int_{0}^{T} \int_{U} \norm{R(r,u)}^{2}_{\mathcal{L}_{2}(\Phi'_{q_{r,u}},\Psi'_{\varrho_{n}})} \mu(du) dr  \geq k \right\}. 
\end{multline*}
where we make the convention $\inf \emptyset = \infty$. Now, by the definition of $\sigma_{n}$ and $\varrho_{n}$, we have
\begin{multline*}
\Prob (T \wedge \sigma_{n} \wedge \tau_{n,k}) 
\leq \frac{3}{k} \left[ \Exp \left( \varrho'_{n} (X_{\sigma_{n} \wedge T} )+  \int_{0}^{\sigma_{n} \wedge T} \, \varrho_{n}'(B(r)) dr \right. \right. \\
\left. \left. +
\int_{0}^{\sigma_{n} \wedge T} \int_{U} \norm{R(r,u)}^{2}_{\mathcal{L}_{2}(\Phi'_{q_{r,u}},\Psi'_{\varrho_{n}})} \mu(du) dr \right)  \right] 
\end{multline*}
Choosing $k=m_{n}$ large enough we conclude that 
\begin{equation*}
\Prob (T \wedge \sigma_{n} \wedge \tau_{n,k})  < 2^{-n}. 
\end{equation*}
Then, by an standard application of the Borel-Cantelli lemma we have $\Prob$-a.e. $\tau_{n,k} =T \wedge \sigma_{n}$ for sufficiently large $n$. 

For every $n \in \N$, let $\tau_{n}\defeq \tau_{n,m_{n}} \wedge \tau_{n+1,m_{n+1}} \wedge \dots $. Then $\tau_{n} \nearrow T$ and $\tau_{n}=T$ for sufficiently large $n$ $\Prob$-a.e. By construction, on $[0,\tau_{n}]$ the process $X$ is bounded in $\Psi'_{\varrho_{n}}$, thus $X$ is bounded in $\Psi'$. 

Then, our Step 1 shows that on the interval $[0,\tau_{n}]$ there exists $C>0$ such that $\abs{F(t,X_{t-})}$, $\abs{F_{t}(t,X_{t-})}$,   $\norm{ F_{x}(t,X_{t-})}_{\mathcal{L}(\Psi'_{\varrho_{n}},\R)}$, $\norm{F_{xx}(t,X_{t-})}_{\mbox{Bil}(\Psi'_{\varrho_{n}}, \Psi'_{\varrho_{n}})}$ are all bounded by $C$. Then, 
$ \int_{0}^{t} \mathbbm{1}_{[0,\tau_{n}]} F_{t}(s,X_{s-}) ds $ is a well-defined Lebesgue integral $\omega$-wise $\Prob$-a.e. 

Likewise, since on  $[0,\tau_{n}]$ we have
\begin{equation}\label{eqNormBoundedByC}
\begin{aligned}
\abs{F_{x}(t,X_{t-}(\omega))B(t,\omega)} \leq C  \varrho'_{n}(B(t,\omega)).\\
 \norm{F_{x}(t,X_{t-}(\omega))R(t,\omega, u)}_{\mathcal{L}(\Psi'_{q_{r,u}},\R)} \leq C \norm{R(r,u)}_{\mathcal{L}_{2}(\Phi'_{q_{r,u}},\Psi'_{\varrho_{n}})}. \\
\abs{ \mbox{Tr}_{R(t,\omega,u)}(F_{xx}(t,X_{t-}))} 
\leq C \norm{R(r,u)}_{\mathcal{L}_{2}(\Phi'_{q^{c}_{r,u}},\Psi'_{\varrho_{n}})}. 
\end{aligned}
\end{equation}
Therefore, the stochastic integral $ \int_{0}^{t} \int_U \,\mathbbm{1}_{[0,\tau_{n}]} F_{x}(s, X_{s-})R(s, u) \, M(ds, du) $ is a well-defined real-valued square-integrable martingale, and $\Prob$-a.e. the  path-wise Lebesgue integrals $\int_{0}^{t} \,\mathbbm{1}_{[0,\tau_{n}]}  F_{x}(s, X_{s-}) B(s) \, ds $ and $
 \int_{0}^{t} \!\!\int_U  \mathbbm{1}_{[0,\tau_{n}]} \mbox{Tr}_{R(s, u)} \left( F_{x x}(s, X_{s-}) \right) \mu(du) ds$ are well-defined.
 
 From the localization properties of the Lebesgue and stochastic integrals, we conclude that on $[0,T]$ the integrals in  It\^{o}'s formula are all well-defined and do not depend on the choice of $\tau_{n}$, $\sigma_{n}$ and $\varrho_{n}$. 
 
On the other hand, below as part of our proof of the convergence in \eqref{eqConverSumsItoFormula} we will show that 
$$ \sum_{0<s \leq t}\left[F\left(s,X_s\right)-F\left(s,X_{s-}\right)-F_x\left(s,X_{s-}\right)\left(\Delta X_{s} \right)\right] <\infty,$$
for every $t \in [0,T]$ $\Prob$-a.e. 
 
Now, put 
 $$ X^{n} = \xi^{n} +\int_{0}^{t} B^{n}(r)dr+ \int_{0}^{t} \int_{U} R^{n}(r,u) M(dr,du),$$
 where $\xi^{n}(\omega)=\xi(\omega) \mathbbm{1}_{\{\varrho'_{1}(\xi)\leq n\}}(\omega)$, and 
 $$B^{n}(r,\omega)=\mathbbm{1}_{[0,\tau_{n}(\omega)]}(r)B(r,\omega), \quad R^{n}(r,\omega, u)=\mathbbm{1}_{[0,\tau_{n}(\omega)]}(r)R(r,\omega,u).$$ 
 In the above, recall from the proof of Theorem \ref{theoLocalHilbertSemima} that $\xi$ has a $\Psi'_{\varrho_{1}}$-valued version $\tilde{\xi}$ satisfying $\Exp \left(\varrho'_{1}(\tilde{\xi})\right)<\infty$.

If It\^{o}'s formula holds for $X^{n}$, this as in the statement of Step 2, then we have 
$$
\begin{aligned}
& F(t, X^{n}_t) = F(0,\xi^{n})+\int_{0}^{t} F_{t}(s, X^{n}_{s-}) d s + \int_{0}^{t}\!\!\int_U \, F_{x}(s, X^{n}_{s-})R^{n}(s, u) \, M(ds, du) \\
& + \int_{0}^{t} \, F_{x}(s, X^{n}_{s-}) B^{n}(s) \, ds  +\frac{1}{2} \int_{0}^{t} \!\!\int_U \textup{Tr}_{R^{n}(s, u)} \left( F_{x x}(s, X^{n}_{s-}) \right) \mu(du) ds \\
& +\sum_{0<s \leq t}\left[F\left(s,X^{n}_s\right)-F\left(s,X^{n}_{s-}\right)-F_x\left(s,X^{n}_{s-}\right)\left(\Delta X^{n}_{s} \right)\right].
\end{aligned}
$$
We must show that both sides in the above equality converges (at least at subsequence) almost surely to the corresponding sides in the It\^{o}'s formula. 

First, observe that by the definition of $\tau_{n}$,  \eqref{eqLocalDriftIntegral}, \eqref{eqLocalDiffusionIntegral}, and because for every $n \in \N$,  $i'_{\varrho_{n}} \in \mathcal{L}(\Psi'_{_{\varrho_{n}}}, \Psi')$, we have that $\Prob$-a.e. $X^{n}_{t} \rightarrow X_{t}$ in $\Psi'$. The continuity of $F$ shows $F(t,X^{n}_{t}) \rightarrow F(t,X_{t})$ $\Prob$-a.e. Also, it is clear that $F(0,\xi^{n}) \rightarrow F(0,\xi)$ $\Prob$-a.e

 Now, since $\mathbbm{1}_{]\tau_{n},T]}(s) \abs{F_{t}(s,X_{s-})} \leq C$, and because
$$ \int_{0}^{t} \abs{F_{t}(s,X^{n}_{s-}) -F_{t}(s,X_{s-})} ds  = \int_{0}^{t} \mathbbm{1}_{]\tau_{n},T]} \abs{F_{t}(s,X_{s-})} ds  $$
by monotone convergence we have $\int_{0}^{t} F_{t}(s, X^{n}_{s-}) d s \rightarrow \int_{0}^{t} F_{t}(s, X_{s-}) d s$ $\Prob$-a.e. Likewise, by the bounds \eqref{eqNormBoundedByC} we can show $\Prob$-a.e.
$$ \int_{0}^{t} \, F_{x}(s, X^{n}_{s-}) B^{n}(s) \, ds   \rightarrow \int_{0}^{t} \, F_{x}(s, X_{s-}) B(s) \, ds. $$
$$ \int_{0}^{t} \!\!\int_U \textup{Tr}_{R^{n}(s, u)} \left( F_{x x}(s, X^{n}_{s-}) \right) \mu(du) ds \rightarrow
\int_{0}^{t} \!\!\int_U \textup{Tr}_{R(s, u)} \left( F_{x x}(s, X_{s-}) \right)  \mu(du)ds. $$
Furthermore, by \eqref{eqNormBoundedByC}  
\begin{flalign*}
& \Exp \int_{0}^{T} \int_{U}   \norm{F_{x}(t,X^{n}_{t-})R^{n}(t, u)-F_{x}(t,X_{t-})R(t, u)}_{\mathcal{L}(\Psi'_{q_{r,u}},\R)}^{2}  \mu(du) dr  \\
& \leq C^{2} \Exp  \int_{0}^{T} \int_{U} \mathbbm{1}_{]\tau_{n},T]}  \norm{R(r,u)}^{2}_{\mathcal{L}_{2}(\Phi'_{q_{r,u}},\Psi'_{\varrho_{n}})} \mu(du) dr. 
\end{flalign*}
Thus, by dominated convergence  
$$  \Exp \int_{0}^{T} \int_{U}   \norm{F_{x}(t,X^{n}_{t-})R^{n}(t, u)-F_{x}(t,X_{t-})R(t, u)}_{\mathcal{L}(\Psi'_{q_{r,u}},\R)}^{2}  \mu(du) dr  \rightarrow 0. $$
Hence by the continuity of the (weak) stochastic integral (Theorem 4.7 in \cite{FonsecaMora:SPDECMVM}) we have
$$ \int_{0}^{t} \int_U \, F_{x}(s, X^{n}_{s-})R^{n}(s, u) \, M(ds, du) 
\rightarrow \int_{0}^{t} \int_U \, F_{x}(s, X_{s-})R(s, u) \, M(ds, du), $$
in the space $\mathcal{M}^{2}_{T}$ of real-valued square integrable martingales on the time interval $[0,T]$. Therefore, one can extract a subsequence such that the convergence is almost surely for all $t \in [0,T]$.  

Finally, we must show
\begin{multline}\label{eqConverSumsItoFormula}
 \sum_{0<s \leq t}\left[F\left(s,X^{n}_s\right)-F\left(s,X^{n}_{s-}\right)-F_x\left(s,X^{n}_{s-}\right)\left(\Delta X^{n}_{s} \right)\right] \\
 \rightarrow \sum_{0<s \leq t}\left[F\left(s,X_s\right)-F\left(s,X_{s-}\right)-F_x\left(s,X_{s-}\right)\left(\Delta X_{s} \right)\right].
\end{multline}
In effect, by Taylor's theorem  (see \cite{Yamamuro:1974}, p.28) for every $s$ there exists $\theta \in (0,1)$ such that 
$$\abs{F\left(s,X^{n}_s\right)-F\left(s,X^{n}_{s-}\right)-F_x\left(s,X^{n}_{s-}\right)\left(\Delta X^{n}_{s} \right)}
\leq \frac{1}{2} \abs{F_{xx}\left(s,X^{n}_{s}+\theta \Delta X^{n}_{s} \right)( X^{n}_{s}, X^{n}_{s}) }.$$
Since $ \mathbbm{1}_{[0,\tau_{n}]} \norm{F_{xx}(t,X_{t-})}_{\mbox{Bil}(\Psi'_{\varrho_{n}}, \Psi'_{\varrho_{n}})} \leq C$, then we have (see Remark \ref{remarkQuadraticVariation})
\begin{multline*}
  \sum_{0<s \leq t} \abs{ F\left(s,X^{n}_s\right)-F\left(s,X^{n}_{s-}\right)-F_x\left(s,X^{n}_{s-}\right)\left(\Delta X^{n}_{s} \right)} \\ \leq \frac{C}{2} \sum_{0<s \leq t}  \varrho'_{n}\left(\Delta X^{n}_{s} \right)^{2}  \leq \frac{C}{2} [X]_{t} < \infty.   
\end{multline*}
By the Wierstrass $M$-test we obtain the convergence in \eqref{eqConverSumsItoFormula} $\Prob$-a.e. This proves that It\^{o}'s formula holds true under the assumptions of Step 2. 

\textbf{Step 3. }  We prove It\^{o}'s formula under Assumption A. 
It suffices to show that under  the conditions in Assumption A we are in a special case of the It\^{o}'s formula for cylindrical martingale-valued measures in a Hilbert space proved in \cite{CambroneroEtAl:ItoFormula}.

Let $\rho$ be a continuous Hilbertian seminorm on $\Psi$ such that $q_{r,u} \leq \rho$, $\forall r \in [0,T]$, $u \in U$, and 
$\int_{0}^{T} \int_{U} q_{r,u}(\cdot)^{2} \mu(du)dr \leq \varrho(\cdot)^{2}$. The existence of such a $\rho$ follows from the properties of the seminorms $q_{r,u}$ (see Definition \ref{defiCMVM}) and because $\Phi$ is nuclear and barrelled (see the proof of Theorem 4.2 in \cite{BojdeckiJakubowski:1989}). The above shows that the mapping $M(t,A):\Phi_{\rho} \rightarrow L^{2} \ProbSpace$ is continuous. Hence, one can check that $M$ defines an orthogonal cylindrical martingale-valued measure on the Hilbert space $\Phi_{\rho}$ (see Definition 4.1 in \cite{CambroneroEtAl:SPDE}). Indeed, as explained in Example 5.23 in \cite{CambroneroEtAl:SPDE}, the family of intensity measures associated to $M$ is of the form:
$$ \nu_{\phi}(\omega)(C)=\int_{C} q_{r,u}(\phi)^{2} \mu(du)dr, \quad \, \forall \, \phi \in \Phi_{\rho}, $$
 and its quadratic variation is 
$$ \quadraVari{M}(C)=\sup_{ n \geq 1} \nu_{\phi_{n}}(\omega)(C)=  \int_{C} \sup_{n \geq 1} q_{r,u}(\phi_{n})^{2} \mu(du)dr, $$
$\forall \, C \in \mathcal{B}(\R_{+}) \otimes \mathcal{B}(U)$ and for $(\phi_{n})_{n \geq 1}$  dense in the unit sphere in $\Phi_{\rho}$. Obverse that if $i_{q_{r,u},\rho}: \Phi_{\rho} \rightarrow \Phi_{q_{r,u}}$ is the canonical inclusion, then 
\begin{equation}\label{eqQuadraVariation}
\quadraVari{M}(C)= \int_{C} \norm{i_{q_{r,u},\rho}}_{\mathcal{L}(\Phi_{\rho},\Phi_{q_{r,u}})}^{2} \mu(du)dr. 
\end{equation}

Moreover, the operator-valued quadratic covariation operator $Q_{M}:[0,T] \times \Omega \times U \rightarrow \mathcal{L}(\Phi_{\varrho}, \Phi_{\varrho})$ is given by 
$$ \rho ( Q_{M}(r,\omega,u)\phi, \varphi)=\frac{q_{r,u}(\phi,\varphi)}{\norm{i_{q_{r,u},\rho}}^{2}_{\mathcal{L}(\Phi_{\rho},\Phi_{q_{r,u}})}}, \quad \forall \, \phi, \varphi \in \Phi_{\rho}. $$
Recall that $Q_{M}$ is non-negative, self-adjoint and is such that $\forall \, 0 \leq t \leq T$, $A \in \mathcal{A}$, $\phi, \varphi \in \Phi_{\rho}$, $\Prob$-a.e. 
$$ \inner{M(A)(\phi)}{M(A)(\varphi)}_{t}(\omega)=\int_{0}^{t} \int_{A} \rho ( Q_{M}(r,\omega,u)\phi, \varphi) \quadraVari{M}(\omega)(dr,du). $$
We must check that our assumptions are enough to  conclude that $X$ is an It\^{o} process in the Hilbert space sense (see Definition 4.1 in \cite{CambroneroEtAl:ItoFormula}). In particular, we must check that $R$ is stochastically integrable with respect to $M$ in the Hilbert space sense (see Section 6 in \cite{CambroneroEtAl:SPDE}).

 In effect, observe that if $(\psi^{\eta}_{n}: n \in \N) \subseteq \Phi$ is a complete orthonormal system in $\Psi_{\eta}$, then by the definition of $Q_{M}$ and $\quadraVari{M}$ we have  
 \begin{flalign*}
& \Exp \int_{0}^{T} \int_{U} \norm{R(r,u)}^{2}_{\mathcal{L}_{2}(\Phi'_{q_{r,u}},\Psi'_{\eta})} \mu(du)dr \\
& =\Exp \int_{0}^{T} \int_{U} \sum_{n=1}^{\infty} q_{r,u}(R(r,u)'\psi^{\eta}_{n})^{2} \mu(du)dr \\
& =\Exp \int_{0}^{T} \int_{U} \sum_{n=1}^{\infty} \rho(Q_{M}(r,u)^{1/2}  R(r,u)'\psi^{\eta}_{n})^{2} \quadraVari{M}(dr,du) \\
& =\Exp \int_{0}^{T} \int_{U} \norm{R(r,u)Q'_{M}(r,u)^{1/2}}^{2}_{\mathcal{L}_{2}(\Phi'_{\rho},\Psi'_{\eta})} \quadraVari{M}(dr,du) 
 \end{flalign*}
  
By Assumption A we conclude that $R \in \Lambda^{2}(M,T;\Phi'_{\rho},\Psi'_{\eta})$, i.e. $R$ is stochastically integrable with respect to $M$ in the Hilbert space sense (Definition   6.4 in \cite{CambroneroEtAl:SPDE}). 
  
\textbf{Claim:}
$\int_{0}^{t} \int_{U} R(r,u) M(dr,du)$ coincides with the stochastic integral in the Hilbert space sense. 

To prove the Claim, let $(S^{n}: n \in \N)$ be a sequence of simple processes converging to $R$ in $\Lambda^{2}(M,T;\Phi'_{\rho},\Psi'_{\eta})$. Denote by $I(S^{n})$ the stochastic integral in the Hilbert space sense for $S^{n}$ and $I(R)$ the stochastic integral for $R$. Then, $I(S^{n}) \rightarrow I(R)$ in the space $\mathcal{M}^{2}_{T}(\Psi'_{\eta})$. 

On the other hand, for every $\psi \in \Psi_{\eta}$ we have
\begin{flalign*}
& =\Exp \int_{0}^{T} \int_{U} q_{r,u}(S^{n}(r,u)'\psi - R(r,u)'\psi)^{2} \mu(du)dr \\
& =\Exp \int_{0}^{T} \int_{U} \rho(Q_{M}(r,u)^{1/2}(S^{n}(r,u)'\psi-  R(r,u)'\psi)^{2} \quadraVari{M}(dr,du)\\
& \leq \eta(\psi)^{2} \Exp \int_{0}^{T} \int_{U} \norm{(S(r,u)-R(r,u))Q'_{M}(r,u)^{1/2}}^{2}_{\mathcal{L}_{2}(\Phi'_{\rho},\Psi'_{\eta})} \quadraVari{M}(dr,du) \\
& \rightarrow 0 
\end{flalign*}
as $n \rightarrow \infty$. Therefore, by Theorem 4.7 in \cite{FonsecaMora:SPDECMVM} we have  
$$ \int_{0}^{t} \int_{U} S^{n}(r,u)' \psi M(dr,du) \rightarrow  \int_{0}^{t} \int_{U} R(r,u)' \psi M(dr,du), $$
in $\mathcal{M}^{2}_{T}(\R)$. 
Assume for the moment that we have 
$$ \int_{0}^{t} \int_{U} S^{n}(r,u)' \psi M(dr,du) = \inner{I_{t}(S^{n})}{\psi}, \quad \forall \, n \in \N. $$
Since we also have $\inner{I(S^{n})}{\psi} \rightarrow \inner{I(R)}{\psi}$ in  $\mathcal{M}^{2}_{T}(\R)$, 
by uniqueness of limits we conclude that 
$$ \inner{I_{t}(R)}{\psi} = \int_{0}^{t} \int_{U} R(r,u)' \psi M(dr,du) = \inner{\int_{0}^{t} \int_{U} R(r,u) M(dr,du)}{\psi}. $$
Hence, $I_{t}(R)$ and $\int_{0}^{t} \int_{U} R(r,u) M(dr,du)$ are indistinguishable $\Psi'_{\eta}$-valued square integrable martingales. 

It only remains to check that the integrals coincide for simple integrands. In effect, let $S(r,\omega,u)=\mathbbm{1}_{F}(\omega)\mathbbm{1}_{]s,t]}(r)\mathbbm{1}_{A}(u)B$,  for $B \in \mathcal{L}_{2}(\Phi'_{\rho}, \Psi'_{\eta})$. Then for $\psi  \in \Psi_{\eta}$, we have $\Prob$-a.e. 
$$ \int_{0}^{t} \int_{U} S(r,u)' \psi M(dr,du) = \mathbbm{1}_{F} M((s,t],A)(B'\psi)=I_{t}(\inner{S^{n})}{\psi})=\inner{I(S^{n})}{\psi}. $$
This proves the claim. 

By what we have shown, $X$ is an It\^{o} process in the Hilbert space sense, therefore we can apply It\^{o} formula to it (Theorem 4.2 in \cite{CambroneroEtAl:ItoFormula}). 
Since the Lebesgue and stochastic integrals coincide, it only remains to show that
\begin{multline}\label{eqIntegralTrace}
  \int_{0}^{t} \!\!\int_U \textup{Tr}_{R(s, u)} \left( F_{x x}(s, X_{s-}) \right) \mu(du) ds \\
  = \int_{0}^{t} \!\!\int_U \widetilde{\textup{Tr}}_{R(s, u)Q'_{M^{c}}(r,u)^{1/2}} \left( F_{x x}(s, X_{s-}) \right) \quadraVari{M}(ds,du).
\end{multline}
where for any complete orthonormal system $(g_{n}: n \in \N)$ in $\Phi'_{\rho}$ we have 
\begin{multline}\label{eqDefHilbertTrace}
\widetilde{\textup{Tr}}_{R(s, u)Q'_{M^{c}}(r,u)^{1/2}} \left( F_{x x}(s, X_{s-}(\omega)) \right) \\
= \sum_{n =1}^{\infty} F_{x x}(s, X_{s-})(R(s,\omega, u)Q'_{M^{c}}(r,u)^{1/2} g_{n}, R(s,\omega, u)Q'_{M^{c}}(r,u)^{1/2} g_{n}). 
\end{multline}
To prove \eqref{eqIntegralTrace}, observe that because $\Phi$ is nuclear we can assume without loss of generality that $\rho$ is such that the inclusion $i_{q^{c}_{r,u},\rho}: \Phi_{\rho} \rightarrow \Phi_{q^{c}_{r,u}}$ is Hilbert-Schmidt for all $(r,u) \in [0,T] \times U$.  This assumption implies that $Q'_{M^{c}}(r,u)^{1/2} \in \mathcal{L}(\Phi'_{\rho}, \Phi'_{q^{c}_{r,u}})$, which in particular shows that \eqref{eqDefHilbertTrace} is well-defined because $R(s,\omega,u) Q'_{M^{c}}(r,u)^{1/2} \in \mathcal{L}_{2}(\Phi'_{\rho}, \Phi'_{q^{c}_{r,u}})$. 

To prove that $Q'_{M^{c}}(r,u)^{1/2} \in \mathcal{L}_{2}(\Phi'_{\rho}, \Phi'_{q^{c}_{r,u}})$, denote by $\alpha_{\rho}: \Phi_{\rho} \rightarrow \Phi'_{\rho}$ and $\alpha_{q^{c}_{r,u}}: \Phi_{q^{c}_{r,u}} \rightarrow \Phi'_{q^{c}_{r,u}}$ the Riesz isometries between these Hilbert spaces and their duals. We have for $\phi, \varphi \in \Phi_{\rho}$, 
\begin{eqnarray*}
q^{c}_{r,u}(i_{q^{c}_{r,u},\rho} \, \phi, i_{q^{c}_{r,u},\rho} \, \varphi)
& = &  \alpha_{q^{c}_{r,u}} \circ i_{q^{c}_{r,u},\rho} \,  \phi \left( i_{q^{c}_{r,u},\rho} \varphi \right) \\
& = & i'_{q^{c}_{r,u},\rho} \circ \alpha_{q^{c}_{r,u}} \circ i_{q^{c}_{r,u},\rho} \, \phi (\varphi) \\
& = & \rho \left(  \alpha_{\rho}^{-1} \circ i'_{q^{c}_{r,u},\rho} \circ  \alpha_{q^{c}_{r,u}} \circ i_{q^{c}_{r,u},\rho} \phi, \varphi \right)
\end{eqnarray*}
This shows 
$$ Q_{M^{c}}(r,u) = \frac{ \alpha_{\rho}^{-1} \circ i'_{q^{c}_{r,u},\rho} \circ \alpha_{q^{c}_{r,u}} \circ i_{q^{c}_{r,u},\rho}}{\norm{i_{q_{r,u},\rho}}^{2}_{\mathcal{L}(\Phi_{\rho},\Phi_{q_{r,u}})}}. $$ 
Since the composition of Hilbert-Schmidt operators is nuclear, we have $Q_{M^{c}}(r,u)$ is a nuclear operator. The dual operator, 
$$ Q'_{M^{c}}(r,u) = \frac{ i'_{q^{c}_{r,u},\rho} \circ \alpha_{q^{c}_{r,u}} \circ i_{q^{c}_{r,u},\rho} \circ   \alpha_{\rho}^{-1} }{\norm{i_{q_{r,u},\rho}}^{2}_{\mathcal{L}(\Phi_{\rho},\Phi_{q_{r,u}})}},$$ 
is also nuclear. Consider the operator
$$ \tilde{Q}(r,u) \defeq \norm{i_{q_{r,u},\rho}}^{2}_{\mathcal{L}(\Phi_{\rho},\Phi_{q_{r,u}})} Q'_{M^{c}}(r,u) = i'_{q^{c}_{r,u},\rho} \circ \alpha_{q^{c}_{r,u}} \circ i_{q^{c}_{r,u},\rho} \circ   \alpha_{\rho}^{-1}. $$
Then $\tilde{Q}(r,u)$ is nuclear, non-negative and symmetric. Following the same ideas as in the proof of Theorem 2.5 in \cite{BojdeckiJakubowski:1990}, Step 3, p.198, one can show that $ \tilde{Q}(r,u)^{1/2}$ is onto $\Phi'_{q^{c}_{r,u}}$  and is continuous as an operator from $\Phi'_{\rho}$ into $\Phi'_{q^{c}_{r,u}}$ (this implies $Q'_{M^{c}}(r,u)^{1/2}$ has this properties as well). Moreover,  $ \tilde{Q}(r,u)^{1/2}$ takes the form 
$$ \tilde{Q}(r,u)^{1/2} f=\sum_{n=1}^{\infty} \lambda_{n}(r,u)e_{n}(r,u)\rho'(f,e_{n}(r,u)), $$
where for each $(r,u)$ we have $\lambda_{1}(r,u) \geq \lambda_{2}(r,u) \geq \cdots >0$, $(e_{n}(r,u): n \in \N)$ is an orthonormal system in $\Phi'_{\rho}$ with  $e_{n}(r,u) \in \Phi'_{q^{c}_{r,u}}$. Furthermore, $(\lambda_{n}(r,u)e_{n}(r,u):n \in \N)$ is an orthonormal basis in $\Phi'_{q^{c}_{r,u}}$.  Therefore, 
\begin{flalign*}
& \mbox{Tr}_{R(s, u)} \left( F_{x x}(s, X_{s-}) \right) \\
& = \sum_{n=1}^{\infty} F_{x x}(s, X_{s-})  (R(s,\omega, u)\lambda_{n}(r,u)e_{n}(r,u), R(s, \omega,u)\lambda_{n}(r,u)e_{n}(r,u))\\
& = \sum_{n =1}^{\infty} F_{x x}(s, X_{s-})(R(s, u)\tilde{Q}(r,u)^{1/2} e_{n}(r,u), R(s, u)\tilde{Q}(r,u)^{1/2} e_{n}(r,u))\\
& = \widetilde{\textup{Tr}}_{R(s, u)Q'_{M^{c}}(r,u)^{1/2}} \left( F_{x x}(s, X_{s-}) \right) \norm{i_{q_{r,u},\rho}}^{2}_{\mathcal{L}(\Phi_{\rho},\Phi_{q_{r,u}})} 
\end{flalign*}
The above calculation together with \eqref{eqQuadraVariation} shows that \eqref{eqIntegralTrace} holds true. This finalizes our proof of Step 3, hence of It\^{o}'s formula. 

\end{proof}

\begin{example} Following \cite{BojdeckiGorostiza:1986, BojdeckiJakubowski:1990}, a $\Phi'$-valued adapted continuous mean-zero Gaussian process $W=( W_{t} : t \geq 0)$ is called a \emph{generalized Wiener process} if 
\begin{enumerate}
\item $W_{t}-W_{s}$ is independent of $\mathcal{F}_{s}$, for $0 \leq s < t$, 
\item 
\begin{equation*} \label{covarianceFunctionalGeneralizedWiener}
\Exp \left( \inner{W_{t}}{\phi} \inner{ W_{s}}{\varphi} \right) = \int_{0}^{t \wedge s} q_{r}(\phi,\varphi)dr , \quad \forall \, t, s \in R_{+}, \, \phi \in \Phi. 
\end{equation*} 
\end{enumerate} 
where $\{q_{r}: r \in \R_{+}\}$ is a family of continuous Hilbertian semi-norms on $\Phi$, such that the map $r \mapsto q_{r}(\phi,\varphi)$ is Borel measurable and bounded on finite intervals, for each $\phi$, $\varphi$ in $\Phi$. Is is easy to verify that 
\begin{equation*} \label{generalizedWienerMartVM}
M(t,A)=W_{t} \delta_{0}(A), \quad \forall \, t \in \R_{+}, \, A \in \mathcal{B}(\{ 0 \}), 
\end{equation*}
defines a cylindrical martingale-valued measure (orthogonal, with independent increments and nuclear, here $U=\{0\}$ and $\mu=\delta_{0}$). Since $W$ has continuous paths, we have $M^{c}(t,A)=M(t,A)$ and $M^{d}(t,A)=0$. 


If $R=(R_{t}: t \leq 0)$ such that $R_{t}(\omega) \in \mathcal{L}(\Phi'_{q_{r}},\Psi')$, satisfies Assumption \ref{assumItoProcess}(3)-(4), then one can show $R$ is stochastically integrable with respect to $W$ in the sense of 
\cite{BojdeckiJakubowski:1990}, and
$$\int_{0}^{t} \int_{U} R_{r} \, M(dr,du) = \int_{0}^{t} R_{r} \,  dW_{r}.$$
If $X$ is an It\^{o} process as in Definition \ref{defiItoProcess}, then  $X$ has continuous paths.  Then, It\^{o}'s formula in Theorem \ref{theoItoFormula} takes the form
$$
\begin{aligned}
& F(t, X_t) = F(0,\xi)+\int_{0}^{t} F_{t}(s, X_{s-}) d s + \int_{0}^{t} \, F_{x}(s, X_{s-})R_{s} \, dW_{s} \\
& + \int_{0}^{t} \, F_{x}(s, X_{s-}) B(s) \, ds  +\frac{1}{2} \int_{0}^{t} \! \textup{Tr}_{R_{s}} \left( F_{x x}(s, X_{s-}) \right) ds,
\end{aligned}
$$
which coincides with the result in \cite{Jakubowski:1995}.
\end{example}

\section{It\^{o}'s formula for L\'{e}vy-It\^{o} processes}\label{sectItoFormLevyProcess}

\begin{assumption}
All through this section $\Phi$ denotes a quasi-complete barrelled nuclear space and $\Psi$ is a complete bornological nuclear space.      
\end{assumption}

\subsection{L\'{e}vy processes and their stochastic integrals}

\textbf{L\'evy processes and the L\'evy-It\^{o} decomposition.}
Following \cite{FonsecaMora:Levy}, a $\Phi'$-valued process $L=( L_{t} : t\geq 0)$ is called a \emph{L\'{e}vy process} if 
\begin{enumerate}[label=(\roman*)]
\item  $L_{0}=0$ a.s., 
\item $L$ has \emph{independent increments}, i.e. for any $n \in \N$, $0 \leq t_{1}< t_{2} < \dots < t_{n} < \infty$ the $\Phi'$-valued random variables $L_{t_{1}},L_{t_{2}}-L_{t_{1}}, \dots, L_{t_{n}}-L_{t_{n-1}}$ are independent,  
\item L has \emph{stationary increments}, i.e. for any $0 \leq s \leq t$, $L_{t}-L_{s}$ and $L_{t-s}$ are identically distributed,   
\item For every $t \geq 0$ the distribution $\mu_{t}$ of $L_{t}$ is a Radon measure and the mapping $t \mapsto \mu_{t}$ from $\R_{+}$ into the space $\goth{M}_{R}^{1}(\Phi')$ of Radon probability measures on $\Phi'$ is continuous at $0$ when $\goth{M}_{R}^{1}(\Phi')$  is equipped with the weak topology. 
\end{enumerate}

By Corollary 3.11 in \cite{FonsecaMora:Levy} there exists a weaker countably Hilbertian topology $\vartheta$ on $\Phi$  such that $L$ has a $(\widehat{\Phi}_{\vartheta})'$-valued c\`{a}dl\`{a}g version $\tilde{L}=( \tilde{L}_{t} : t \geq 0)$. In particular, $\tilde{L}$ is a $\Phi'$-valued c\`{a}dl\`{a}g L\'{e}vy process which is a version of $L$. We will identify $L$ with $\tilde{L}$.

Let $N=\{N(t,A): \, t \geq 0, A \in \mathcal{B}(\Phi' \setminus \{ 0\})\}$ be the \emph{Poisson random measure} associated to $L$, i.e. $ N(t,A)= \sum_{0 \leq s \leq t} \ind{A}{\Delta L_{s}}$, $\forall \, t \geq 0$, $A \in \mathcal{B}( \Phi' \setminus \{ 0\})$, with respect to the ring $\mathcal{A}$ of all the subsets of $\Phi' \setminus \{0\}$ that are \emph{bounded below} (i.e. $A \in \mathcal{A}$ if $0 \notin \overline{A}$, where $\overline{A}$ is the closure of $A$). The corresponding compensator measure of $N$ is of the form $\nu(\omega; dt; df)= dt \nu (df)$, where $\nu$ is a \emph{L\'{e}vy measure} on $\Phi'$ in the following sense (see \cite{FonsecaMora:Levy}, Theorem 4.11):
\begin{enumerate}
\item $\nu (\{ 0 \})=0$, 
\item for each neighborhood of zero $U \subseteq \Phi'$, the  restriction $\restr{\nu}{U^{c}}$ of $\nu$ on the set $U^{c}$ belongs to the space $\goth{M}^{b}_{R}(\Phi')$ of bounded Radon measures on $\Phi'$,    
\item there exists a continuous Hilbertian semi-norm $\rho$ on $\Phi$ such that 
\begin{equation} \label{integrabilityPropertyLevyMeasure}
\int_{B_{\rho'}(1)} \rho'(f)^{2} \nu (df) < \infty,  \quad \mbox{and} \quad  \restr{\nu}{B_{\rho'}(1)^{c}} \in \goth{M}^{b}_{R}(\Phi'), 
\end{equation}
where  $\rho'$ is the dual norm of $\rho$ and $B_{\rho'}(1) \defeq \{ f \in \Phi': \rho'(f) \leq 1\}$. 
\end{enumerate}

It is shown in Theorem 4.17 in \cite{FonsecaMora:Levy} that relative to $\rho$ satisfying \eqref{integrabilityPropertyLevyMeasure},  for each $t \geq 0$, $L_{t}$ admits the following the \emph{L\'{e}vy-It\^{o} decomposition}:
\begin{equation} \label{eqLevyItoDecomposition}
L_{t}=t\goth{m}+W_{t}+\int_{B_{\rho'}(1)} f \widetilde{N} (t,df)+\int_{B_{\rho'}(1)^{c}} f N (t,df).
\end{equation}
In \eqref{eqLevyItoDecomposition}, we have  $\goth{m} \in \Phi'$, $\widetilde{N}(dt,df)= N(dt,df)-dt \, \nu(df)$ is the compensated Poisson random measure, and $( W_{t} : t \geq 0)$ is a $\Phi'$-valued Wiener process with mean-zero  and \emph{covariance functional} $\widetilde{\mathcal{Q}}$ satisfying 
\begin{equation}\label{covarianceFunctWienerProcess}
\Exp \left( \inner{W_{t}}{\phi} \inner{W_{s}}{\varphi} \right) = ( t \wedge s ) \widetilde{\mathcal{Q}} (\phi, \varphi), \quad \forall \, \phi, \varphi \in \Phi, \, s, t \geq 0. 
\end{equation}
Observe that $\widetilde{\mathcal{Q}}$ is a continuous, symmetric, non-negative bilinear form on $\Phi \times \Phi$. The associated Hilbertian seminorm $\mathcal{Q}$ is defined by $\mathcal{Q}(\phi)=\widetilde{\mathcal{Q}}(\phi, \phi)^{1/2}$ $\forall \phi \in \Phi$. The process $\int_{B_{\rho'}(1)} f \widetilde{N} (t,df)$, $t\geq 0$, is a $\Phi'$-valued mean-zero, square integrable, c\`{a}dl\`{a}g L\'{e}vy process with second moments given by $\Exp \left( \abs{ \inner{\int_{B_{\rho'}(1)} f \widetilde{N} (t,df)}{\phi} }^{2}\right) = t \int_{B_{\rho'}(1)} \abs{\inner{f}{\phi}}^{2} \nu (df)$ $\forall \, t \geq 0$ and $\phi \in \Phi$, and the process $\int_{B_{\rho'}(1)^{c}} f N (t,df)$  $\forall t\geq 0$ is a $\Phi'$-valued c\`{a}dl\`{a}g L\'{e}vy process defined  by means of a Poisson integral with respect to the Poisson random measure $N$ of $L$ on the set $B_{\rho'}(1)^{c}$ (see Section 4.1 in \cite{FonsecaMora:Levy} for more information on Poisson integrals in duals of nuclear spaces). All the random components of the representation \eqref{eqLevyItoDecomposition} are independent.


Now we recall from \cite{FonsecaMora:SPDECMVM, FonsecaMora:SPDELevy} the construction of stochastic integrals with respect to L\'evy processes via cylindrical martingale-valued measures. 

\textbf{L\'evy martingale-valued measures.} Let $U \in \mathcal{B}(\Phi')$ be such that $0 \in U$ and $\int_{U} \, \abs{\inner{f}{\phi}}^{2} \nu(df)< \infty$ for every $\phi \in \Phi$. Take  $\mathcal{R}= \{  U \cap \Gamma: \Gamma \in \mathcal{A} \} \cup \{ \{0\}\}$. 
Define $M=(M(t,A): t \geq 0, A \in \mathcal{R})$  by
\begin{equation} \label{levyMartValuedMeasExam} 
M(t,A) = W_{t} \delta_{0}(A) + \int_{A \backslash \{0 \}} f \widetilde{N}(t,df), \quad \mbox{ for } \, t \geq 0, \, A \in \mathcal{R}. 
\end{equation}
Then $M$ is a nuclear cylindrical martingale-valued measure with independent increments. Moreover, for each $0 \leq s < t$, $A \in \mathcal{R}$ we have:
\begin{equation} \label{secondMomentLevyNuclearMartValuedMeasu}
\Exp \left( \abs{\inner{M((s,t],A)}{\phi}}^{2} \right)=(t-s) \left[ \mathcal{Q}(\phi)^{2} + \int_{A \backslash \{0 \}} \abs{\inner{f}{\phi}}^{2}  \nu(df) \right], \quad \forall \, \phi \in \Phi.
\end{equation}
So in comparison with \eqref{eqDefiNuclearCMVM} we have  $\mu = \delta_{0}+ \restr{\nu}{U}$,  
and the family of seminorms $\{q_{r,f}: r \in \R_{+}, f \in U \}$ is given by 
\begin{equation}\label{defiSemiNormsLevyMartValuedMeas}
q_{r,f}(\phi)= \begin{cases} \mathcal{Q}(\phi), & \mbox{if } f=0, \\  \abs{\inner{f}{\phi}}, & \mbox{if } f \in U \setminus \{0\}. \end{cases}
\end{equation}

We call $M$ defined in \eqref{levyMartValuedMeasExam} a \emph{L\'{e}vy martingale-valued measure} based on the set $U$. 

\textbf{Stochastic integration.}
Let $R: [0, \infty) \times \Omega \times \Phi' \rightarrow \mathcal{L}(\Phi',\Psi')$ satisfy that for every $\psi \in \Psi$, $T>0$,  the mapping $(t,\omega,f) \mapsto R(t,\omega,f)'\psi$ is $\mathcal{P}_{T} \otimes \mathcal{B}(\Phi')/\mathcal{B}(\Phi)$-measurable. Assume that for every $T >0$ and $\psi \in \Psi$ the mapping $R$ satisfies 
\begin{multline*}
\Exp \int_{0}^{T} \int_{U} q_{r,f}(R(r,f)'\psi)^{2} \nu(df) dr \\
=\Exp \int_{0}^{T} \, \left[  \mathcal{Q}(R(s,0)'\psi)^{2} + \int_{U} \abs{\inner{f}{R(s,f)'\psi}}^{2} \nu(df)\, \right] ds < \infty.
\end{multline*}
Hence, the stochastic integral $\int^{t}_{0} \int_{U} R (r,f) M (dr, df)$, $t \geq 0$, can be defined using the 
the theory of stochastic integration developed in Section 5 in \cite{FonsecaMora:SPDECMVM} with respect to the L\'{e}vy martingale-valued measure \eqref{levyMartValuedMeasExam}. 

Observe moreover from \eqref{levyMartValuedMeasExam} that we have 
$$ M^{c}(t,A) = W_{t} \delta_{0}(A), \quad M^{d}(t,A)= \int_{A \backslash \{0 \}} f \widetilde{N}(t,df).  $$
Both $M^{c}$ and $M^{d}$ are orthogonal, with independent increments and nuclear. Moreover, from \eqref{defiSemiNormsLevyMartValuedMeas}, for all $r \geq 0$ we have 
$q^{c}_{r,0}=\mathcal{Q}$ and $q^{c}_{r,f}=0$ for $f \in U \setminus \{ 0\}$; and $q^{d}_{r,f}(\phi)=\abs{\inner{f}{\phi}}$ for all $f \in U$.  Hence, from Proposition \ref{propStochIntegralSumContAndDiscParts} we have
\begin{equation}\label{eqDefiIntegralLevyMVM}
\int^{t}_{0} \int_{U} R (r,f) M (dr, df) \defeq \int_{0}^{t}  R(r,0) dW_{r} + \int_{0}^{t} \int_{U}  R(r,f) \widetilde{N}(dr,df), 
\end{equation}
being the Wiener stochastic integral $\int_{0}^{t}  R(r,0) dW_{r}$ the stochastic integral defined with respect to $M^{c}$, and the Poisson stochastic integral $\int_{0}^{t} \int_{U}  R(r,f) \widetilde{N}(dr,df)$ the stochastic integral defined with respect to $M^{d}$. 

Assume that $V \in \mathcal{A}$. For each $t \geq 0$ we define the \emph{Poisson stochastic integral}:
\begin{equation} \label{eqDefiPoissonInteg}
\int_{0}^{t} \int_{V} R(s,f)  N(ds,df)(\omega) 
= \sum_{0 \leq s \leq t} R(s, \omega, \Delta L_{s}(\omega)) \Delta L_{s}(\omega) \ind{V}{\Delta L_{s}(\omega)},
\end{equation}
which is a finite (random) sum. The Poisson stochastic integral \eqref{eqDefiPoissonInteg} is a $\Psi'$-valued regular adapted process. 

If we further assume that  $\int_{0}^{T} \int_{V} \abs{\inner{f}{R(r,f)' \psi}}^{2} \, \nu(df)dr < \infty$ for every $T>0$ and $\psi \in \Psi$. We may then define
\begin{multline}\label{equaDefiCompenPoissInteg}
\int_{0}^{t} \int_{V} \, R(s,f)' \psi \, \widetilde{N}(ds,df) 
\\ = \int_{0}^{t} \int_{V} \, R(s,f)' \psi \,   N(ds,df)- \int_{0}^{t} \int_{V} \inner{f}{R(s,f)' \psi} \, \nu(df)ds.  
\end{multline}
The above definition is consistent with our earlier definition of the compensated Poisson integral defined with respect to the compensated Poisson random measure $M^{d}$. 

Finally, assume that for every $T>0$ and $\psi \in \Psi'$ the mapping $R$ satisfies 
$$ \Exp \int_{0}^{T} \, \abs{\inner{\goth{m}}{R(r, \omega,0)'\psi}} \, dr < \infty, \quad \forall  \, \psi \in \Psi. $$
Then as explained in Section 3.2 in \cite{FonsecaMora:SPDELevy} there exists a $\Psi'$-valued, regular, continuous process $\int_{0}^{t} R(r,0) \goth{m} dr$, $t \geq 0$, satisfying $\Prob$-a.e.
\begin{equation}\label{eqDefiDetermiIntegDriftGeneral}
\inner{\int_{0}^{t} R(r,\omega,0)\goth{m} dr}{\psi}= \int_{0}^{t} \inner{\goth{m}}{R(r,\omega,0)'\psi} ds, \quad \forall \, t \geq 0,  \omega \in \Omega, \psi \in \Psi. 
\end{equation}
Following \cite{FonsecaMora:SPDELevy}, we define the \emph{stochastic integral of $R$ with respect to $L$ (relative to the set $U$)}  as the $\Psi'$-valued  regular adapted c\`{a}dl\`{a}g process defined for each $t \geq 0$ by
\begin{eqnarray}
\int_{0}^{t} \int_{\Phi'} R(s,f) L(ds,df)
& = & \int_{0}^{t}  R(s,0) \goth{m}  ds + \int_{0}^{t} \int_{U}  R(s,f)  M(ds,df)  \nonumber \\
& {} & + \int_{0}^{t} \int_{U^{c}} R(s,f)   N(ds,df). \label{eqDefiLevyIntegral}   
\end{eqnarray} 

We are ready to define the class of processes for which we will prove It\^{o}'s formula.

\begin{definition}\label{defiLevyItoProcess}
Let $L$, $U$, $M$ and $R$ as described in the above paragraphs. Let $\xi$ and $B$ as in Assumption \ref{assumItoProcess}. 
A \emph{L\'{e}vy-It\^{o} process} defined with respect to $L$ is a $\Psi'$-valued process of the form
\begin{equation}\label{eqDefiLevyItoProcess}
X_t = \xi+\int_0^t B(s) \, ds + \int_{0}^{t} \int_{\Phi'} R(s,f) L(ds,df),   
\end{equation} 
with the stochastic integral as in \eqref{eqDefiLevyIntegral}.
\end{definition}

The following is our main result on It\^{o}'s formula for L\'evy-It\^{o} processes. 

\begin{theorem}\label{theoItoFormulaLevyItoProcesses}
Let $L$ be an $\Phi'$-valued c\`adl\`ag L\'{e}vy process with L\'{e}vy-It\^{o} decomposition \eqref{eqLevyItoDecomposition} and assume $\Gamma$ is a separable complete bornological nuclear space. 
Let $X$ be a L\'{e}vy-It\^{o} process of the form \eqref{eqDefiLevyItoProcess} for $U=B_{\rho'}(1)$. Let $F:[0, \infty) \times \Psi' \rightarrow \Gamma$ be a function of class $C^{1,2}$. We assume $F_{t}$ and $F_{xx}$ is uniformly continuous on bounded subsets of $[0,\infty)\times \Psi'$. Then  almost surely we have for every $t\geq 0$
\begin{flalign*}
& F(t, X_t) \\
& = F(0,\xi)+\int_{0}^{t} F_{t}(s, X_{s-}) ds+\int_0^t F_x(s, X_{s-}) R(s,0) \, dW_s \\
& +  \int_{0}^{t} F_x(s, X_{s-}) (B(s)+R(s,0) \goth{m}) \, ds   + \frac{1}{2} \int_{0}^{t}  \mbox{Tr}_{R(s, 0)} \left( F_{x x}(s, X_{s-})\right)  ds  \\ 
   & + \int_{0}^{t}\!\! \int_{B_{\rho'}(1)^c} \left[ F(s,X_{s-}+R(s,f)f) - F(s,X_{s-}) \right] \, N(ds,df) \\
   & +  \int_{0}^{t}\!\! \int_{B_{\rho'}(1)} \left[ F(s,X_{s-}+R(s,f)f) - F(s,X_{s-}) \right] \, \tilde{N}(ds,df) \\
   & +  \int_{0}^{t}\!\! \int_{B_{\rho'}(1)} \left[ F(s,X_{s-}+R(s,f)f) - F(s,X_{s-})-F_{x}(s,X_{s-})(R(s,f)f) \right] \,\nu(df)ds.
\end{flalign*}
\end{theorem}

\begin{proof}
As explained in the proof of Theorem \ref{theoItoFormula}, it suffices to prove the result under the assumption $\Gamma=\R$.

Let $\vartheta$ be a weaker countably Hilbertian topology  on $\Phi$ such that $L$ has a $(\widehat{\Phi}_{\vartheta})'$-valued c\`{a}dl\`{a}g version. Thus we have
$\Prob \left( L_{t} \in (\widehat{\Phi}_{\vartheta})', \forall t \geq 0 \right)=1$. Let $(\rho_{n}: n \in \N)$ an increasing sequence of continuous Hilbertian semi-norms on $\Phi$ that generates the topology of $\vartheta$. Without of generality, we can assume $\rho_{1}=\rho$ (recall $\rho'$ satisfies  \eqref{integrabilityPropertyLevyMeasure}). Since $(\rho_{n})$ is increasing, we have 
(see \cite{GelfandVilenkin}, Section IV.2.2)
$$  (\widehat{\Phi}_{\vartheta})' = \bigcup_{n \in \N} \Phi'_{\rho_{n}}= \bigcup_{n \in \N} B_{\rho'_{n}}(n),$$
where we recall $B_{\rho'_{n}}(n)=\{ f \in \Phi': \rho_{n}'(f) \leq n  \}$. Observe that $\{ B_{\rho_{n}'}(n): n \in \N \}$ is an increasing sequence of  bounded, closed, convex, balanced subsets of $\Phi'$. Furthermore, by the equality above we have  
\begin{equation} \label{levyAlmostSureCountUnionBallsInDual}
\Prob \left( L_{t} \in \bigcup_{n \in \N} B_{\rho'_{n}}(n), \, \forall \, t \geq 0 \right)=1. 
\end{equation}

For each $n \in \N$ let $U_{n}=B_{\rho'_{n}}(n)$ and define $\tau_{n}$ by 
\begin{equation} \label{defiStoppingTimesTauNLevyNoise}
 \tau_{n}(\omega) \defeq \inf \{ t \geq 0: \Delta L_{t} (\omega) \notin U_{n} \}, \quad \forall \, \omega \in \Omega.
\end{equation}  
It is clear that each  $\tau_{n}$ is a stopping time, and from \eqref{levyAlmostSureCountUnionBallsInDual} we have  $\tau_{n} \rightarrow \infty$ $\Prob$-a.e. as $n \rightarrow \infty$. Furthermore, observe that for each $n \in \N$, from the definition of $U_{n}$, for every $\phi \in \Phi$  we have 
\begin{flalign} 
& \int_{U_{n}} \, \abs{\inner{f}{\phi}}^{2} \nu(df) \nonumber \\
&\leq \rho_{1}(\phi)^{2} \int_{B_{\rho'_{1}}(1)} \rho'_{1}(f)^{2} \nu(df)+ n^{2} \rho_{n}(\phi)^{2} \nu (B_{\rho'_{1}}(1)^{c}) < \infty. \label{intLevyMeasuOnUnIsBounded}
\end{flalign}

For every $n \in \N$, consider the stopped process $X^n_t = X_{t \wedge \tau_n}$ for all $t \geq 0$. Clearly $X^{n}_t \rightarrow X_t$ $\Prob$-a.e.  Observe moreover that by the definition of $\tau_{n}$, \eqref{eqDefiIntegralLevyMVM} and \eqref{equaDefiCompenPoissInteg} we have
\begin{align*}
    X^n_t & = \xi+\int_0^{t \wedge \tau_n} (B(s)+R(s,0)\goth{m}) \, ds \\ & \ \ \ + \int_0^{t \wedge \tau_n}\!\!\int_{U_1} \Phi(s,f)\, M(ds, df)  + \int_0^{t \wedge \tau_n}\!\!\int_{U_n \backslash U_1} R(s,f) f \, N(ds, df) \\
    & = \xi+  \int_0^{t \wedge \tau_n} \left( B(s)+ R(s,0)\goth{m} + \int_{U_n \backslash U_1} R(s,f) f \, \nu(df) \right)\, ds \\
    & \ \ \ + \int_0^{t \wedge \tau_n}\!\!\int_{U_1} R(s,f)\, M(ds, df) \\
    & \ \ \ + \int_0^{t \wedge \tau_n}\!\!\int_{U_n \backslash U_1} R(s,f) f \, N(ds, df) - \int_0^{t \wedge \tau_n}\!\! \int_{U_n \backslash U_1} R(s,f) f \, ds \nu(df) \\
    & = \xi+  \int_0^{t \wedge \tau_n} R_{n}(s) \, ds + \int_0^{t \wedge \tau_n}\!\! \int_{U_n} R(s,f)\, M(ds, df)
\end{align*}
In the last line of the above equality, we have 
$$ R_{n}(s)= B(s)+R(s,0)\goth{m} + \int_{U_n \backslash U_1} R(s,f) f \, \nu(df), $$
and $M$ is as in \eqref{levyMartValuedMeasExam} on the set $U_{n}$ (it is well defined due to \eqref{intLevyMeasuOnUnIsBounded}).  This shows that $X^{n}$ is a  generalized It\^{o} process, as in Definition \ref{defiLevyItoProcess} on the set $U_{n}$.

We can then apply It\^o's formula (Theorem \ref{theoItoFormula}) to $X^{n}$ to obtain
\begin{align*}
F(t, X^n_t) & = F(0,\xi)+\int_{0}^{t}\!\!\int_{U_{n}} F_x(s,X^n_{s-}) R(s,f) {M}(ds, df)  + \int_{0}^{t} F_{x}(s, X^n_{s-}) R_n(s) \, ds \\
& \ \ \ + \int_{0}^{t} F_{t}(s, X^n_{s-}) ds  +\tfrac{1}{2} \int_{0}^{t} \mbox{Tr}_{R(s,0)} \left( F_{x x}(s, X^n_{s-}) \right) ds \\
& \ \ \ +\sum_{0<s \leq t}\left[F\left(s,X^n_s\right)-F\left(s,X^n_{s-}\right)-F_x\left(s,X^n_{s-}\right)\left(\Delta X^n_{s} \right)\right].
\end{align*}

By \eqref{eqDefiIntegralLevyMVM} and \eqref{equaDefiCompenPoissInteg}, the first integral corresponds to
\begin{align*}
   & \int_0^t F_x(s, X^n_{s-}) R(s,0) \, dW_s + \int_{0}^{t}\!\!\int_{U_1 \setminus \{0\}} F_x(s,X^n_{s-}) R(s,f)f\, \tilde{N}(ds, df) \\
   +\ & \int_{0}^{t}\!\!\int_{U_n \backslash U_1} F_x(s,X^n_{s-}) R(s,f)f\, {N}(ds, df) - \int_{0}^{t}\!\!\int_{U_n \backslash U_1} F_x(s,X^n_{s-}) R(s,f) f \, \nu(df) \, ds.
\end{align*}
By expanding $R_n$, the second integral of It\^o's formula is equal to
\begin{align*}
    \int_0^t F_x(s, X^n_{s-}) (B(s)+R(s,0) \goth{m}) \, ds + \int_{0}^{t}\!\!\int_{U_n \backslash U_1} F_x(s,X^n_{s-}) R(s,f) f\, \nu(df) \, ds.
\end{align*}
Note that when adding the two decompositions above, the integral with respect to $\nu(df) ds$ cancels out. 

By definition of the Poisson integral, the finite random sum is equal to 
$$
\int_{0}^{t}\!\! \int_{U_n} \left[F(s,X^n_{s-}+R(s,f)f) - F(s,X^n_{s-})-F_{x}(s,X^n_{s-})(R(s,f)f)\right] \, N(ds,df).
$$

Combining all the identities above, we obtain
\begin{align}\label{ItoformulafortruncatedX}
    & F(t, X^n_t) = F(0,\xi)+ \int_{0}^{t} F_{t}(s, X^n_{s-}) ds+\int_0^t F_x(s, X^n_{s-}) R(s,0) \, dW_s   \\
   & + \int_0^t F_x(s, X^n_{s-}) (B(s)+R(s,0) \goth{m}) \, ds + \frac{1}{2} \int_{0}^{t}  \mbox{Tr}_{R(s, 0)} \left( F_{x x}(s, X^n_{s-})\right)  ds \nonumber \\
   & + \int_{0}^{t}\!\!\int_{U_1} F_x(s,X^n_{s-}) R(s,f)f \, \tilde{N}(ds, df) + \int_{0}^{t}\!\!\int_{U_n \backslash U_1} F_x(s,X^n_{s-}) R(s,f)f \, {N}(ds, df)  \nonumber \\
   & + \int_{0}^{t}\!\! \int_{U_n} \, \left[F(s,X^n_{s-}+R(s,f)f) - F(s,X^n_{s-})-F_{x}(s,X^n_{s-})(R(s,f)f) \right] \, N(ds,df). \nonumber
\end{align}
The latest integral can be further decomposed by separating its domain in $U_1$ and $U_n \setminus U_1$, then adding and subtracting the integral with respect to $\nu(du) ds$, to get the compensated Poisson integral on $U_1$, thus obtaining 
\begin{align*}
   & \ \int_{0}^{t}\!\! \int_{U_n \setminus U_1} \left[F(s,X^n_{s-}+R(s,f)f) - F(s,X^n_{s-})-F_{x}(s,X^n_{s-})(R(s,f)f) \right] \, N(ds,df) \\
   & + \int_{0}^{t}\!\! \int_{U_1} \left[F(s,X^n_{s-}+R(s,f)f) - F(s,X^n_{s-})-F_{x}(s,X^n_{s-})(R(s,f)f) \right] \, \tilde{N}(ds,df) \\
   & + \int_{0}^{t}\!\! \int_{U_1} \left[F(s,X^n_{s-}+R(s,f)f) - F(s,X^n_{s-})-F_{x}(s,X^n_{s-})(R(s,f)f) \right] \,\nu(df)ds.
\end{align*}
Note that this computation, along with \eqref{ItoformulafortruncatedX}, provides some extra cancellation for some of the terms, giving us the desired result for the process $X^n$, that is,
\begin{eqnarray*}
F(t, X^n_t) & = & F(0,\xi)+\int_{0}^{t} F_{t}(s, X^n_{s-}) ds + \int_0^t F_x(s, X^n_{s-}) R(s,0) \, dW_s \\
   & + & \int_0^t F_x(s, X^n_{s-}) (B(s)+R(s,0) \goth{m}) \, ds   + \frac{1}{2} \int_{0}^{t}  \mbox{Tr}_{R(s, 0)} \left( F_{x x}(s, X^n_{s-})\right)  ds  \\ 
   & + & \int_{0}^{t}\!\! \int_{U_n \setminus U_1} \left[ F(s,X^n_{s-}+R(s,f)f) - F(s,X^n_{s-}) \right] \, N(ds,du) \\
   & + & \int_{0}^{t}\!\! \int_{U_1} \left[F(s,X^n_{s-}+R(s,f)f) - F(s,X^n_{s-}) \right] \, \tilde{N}(ds,df) \\
   & + & \int_{0}^{t}\!\! \int_{U_1} \left[F(s,X^n_{s-}+R(s,f)f) - F(s,X^n_{s-})-F_{x}(s,X^n_{s-})(R(s,f)f) \right] \,\nu(df)ds.
\end{eqnarray*}
    
By similar arguments to those in Step 2 in the proof of Theorem \ref{theoItoFormula}, we have the following pointwise almost surely convergences 
$$
F(t, X^n_t) \to F(t, X_t), \quad \int_{0}^{t} F_{t}(s, X^n_{s-}) ds \to \int_{0}^{t} F_{t}(s, X_{s-}) ds,
$$
$$
\int_0^t F_x(s, X^n_{s-}) (B(s)+R(s,0) \goth{m}) \, ds \to \int_0^t F_x(s, X_{s-}) (B(s)+R(s,0) \goth{m}) \, ds,
$$
$$
\frac{1}{2} \int_{0}^{t}  \mbox{Tr}_{R(s, 0)} \left( F_{x x}(s, X^n_{s-})\right)  ds \to \frac{1}{2} \int_{0}^{t}  \mbox{Tr}_{R(s, 0)} \left( F_{x x}(s, X_{s-})\right)  ds,
$$
and the latest integral to 
$$
\int_{0}^{t}\!\! \int_{U_1} \left[F(s,X_{s-}+R(s,f)f) - F(s,X_{s-})-F_{x}(s,X_{s-})(R(s,f)f) \right] \,\nu(df)ds.
$$
The same reasoning gives the convergence in probability of stochastic integrals with respect to $W$ and $\tilde{N}$ respectively to
$$
\int_0^t F_x(s, X_{s-}) R(s,0) \, dW_s, \quad  \int_{0}^{t}\!\! \int_{U_1} \left[F(s,X^n_{s-}+R(s,f)f) - F(s,X^n_{s-}) \right] \, \tilde{N}(ds,df).
$$
And finally, the integral with respect to $N$, which is only a finite random sum, converges in probability to 
$$
\int_{0}^{t}\!\! \int_{U_1^{c}} \left[ F(s,X_{s-}+R(s,f)f) - F(s,X_{s-}) \right] \, N(ds,df),
$$
which completes the proof.
\end{proof}

\section{Examples}\label{sectExamples}

In this section we consider examples and applications of the theory developed in the last section.

\begin{example}\label{examSPDEAsLevyIto}
For each $t \geq 0$ let $A(t) \in \mathcal{L}(\Psi,\Psi)$. Assume that for every $\psi \in \Psi$ the mapping $t \mapsto A(t) \psi$ is continuous from $[0,\infty)$ into $\Psi$, and that the family $(A(t): t\geq 0)$ generates the backward evolution system $(U(s,t): 0 \leq s \leq t < \infty)$ (see Section 4.1 in \cite{FonsecaMora:SPDELevy} for definitions and properties of evolution systems).  Consider the following (time-inhomogeneous) stochastic evolution equation, 
\begin{equation}\label{eqSEELevy}
d Y_{t}=A(t)' Y_{t}dt +\int_{\Phi'} R(t,f) L(dt,df), \quad t \geq 0, 
\end{equation}
with initial condition $Y_{0}=\xi$ $\Prob$-a.e. 
By Theorem 4.7 in \cite{FonsecaMora:SPDELevy},  \eqref{eqSEELevy} has a weak solution given by the mild (or evolution) solution 
\begin{equation*}\label{eqDefMildSoluSEELevy}
X_{t}=U(t,0)'\xi+\int_{0}^{t}\int_{\Phi'} \, U(t,s)' R(s,f) \, L(ds,df), \quad \forall \, t \geq 0. 
\end{equation*}
(See Section 4.2 in \cite{FonsecaMora:SPDELevy} for properties of the stochastic convolution.) Since $A(t) \in \mathcal{L}(\Psi,\Psi)$, the mild solution is indeed a strong solution to \eqref{eqSEELevy}, that is 
$$ X_{t}=\xi+\int_{0}^{t} A(s)'X_{s} ds + \int_{0}^{t} \int_{\Phi'} R(s,f) L(ds,df).$$
If we let $B(t,\omega)=A(t)'X_{t}(\omega)$, we might be tempted to think that $X=(X_{t}: t \geq 0)$ is a L\'evy-It\^{o} process, however, $B$ as defined above might not satisfy \eqref{eqLocalIntegCoeffiB}. One need to impose additional assumptions on $L$ to guarantee \eqref{eqLocalIntegCoeffiB} holds. 

Assume additionally that 
\begin{enumerate}[label=(\roman*)]
\item $\xi$ is square integrable, that is, $\Exp \left( \abs{\inner{\xi}{\psi}}^{2}\right)<\infty$ for every $\psi \in \Psi$,
\item $L$ is square integrable, that is, $\Exp \left( \abs{\inner{L_{t}}{\phi}}^{2}\right)<\infty$ for every $t \geq 0$, $\phi \in \Phi$,
\item $R$ satisfies
\begin{equation}\label{eqSquareMomenCoefROnDrif}
\Exp \int_{0}^{T} \, \abs{\inner{\goth{m}}{R(r, \omega,0)'\psi}}^{2} \, dr < \infty, \quad \forall  \, \psi \in \Psi.    
\end{equation} 
\end{enumerate}
Then, by Theorem 5.4 in \cite{FonsecaMora:SPDELevy} the mild solution $X$ is the unique weak solution to  \eqref{eqSEELevy} and for every $T>0$ there exists a continuous Hilbertian seminorm $\varrho$ on $\Psi$ (depending on $T$) such that $\Exp \int_{0}^{T} \varrho'(X_{t})^{2} dt < \infty$. Thus, since the family $(A(t): t\geq 0)$  is locally equicontinuous, there exists $C>0$ (depending on $\varrho$ and $T$) such that $\sup_{0 \leq t \leq T} \varrho(A(t)\psi) \leq C \varrho(\psi)$ for all $\psi \in \Psi$. Hence, for the process $B$ we have for every $T>0$ 
$$ \Exp \int_{0}^{T} \abs{\inner{B(t)}{\psi}}^{2} dt \leq C^{2} \varrho^{2}(\psi) \Exp \int_{0}^{T} \varrho'(X_{t})^{2} dt < \infty, \quad \forall \psi \in \Psi $$
The above implies \eqref{eqLocalIntegCoeffiB}. Therefore, under the additional assumptions above, we have $X$ is a L\'evy-It\^{o} process. 
\end{example}

\begin{example} A particular important instance of Example \ref{examSPDEAsLevyIto} is the stochastic heat equation with additive L\'evy noise in the space of tempered distributions.
Let $\Phi=\Psi=\mathcal{S}(\R^{d})$ and assume $L=(L_{t}:t\geq 0)$ is a $\mathcal{S}'(\R^{d})$-valued  square integrable c\`adl\`ag L\'{e}vy process. Let   $\Delta$ be the Laplace operator on $\mathcal{S}'(\R^{d})$. It is well-known that the Laplace operator $\Delta$ on $\mathcal{S}(\R^{d})$ is the infinitesimal generator of the \emph{heat semigroup}  $(S(t): t \geq 0)$ which is the $C_{0}$-semigroup on $\mathcal{S}(\R^{d})$ defined as: $S(0)=I$ and for each $t>0$, 
\begin{equation*}
(S(t)\psi)(x) \defeq \frac{1}{(4 \pi t)^{d/2}}  \int_{\R^{d}} e^{-\norm{x-y}^{2}/4t} \psi(y) dy, \quad \forall \psi \in  \mathcal{S}(\R^{d}), \, x \in \R^{d}. 
\end{equation*}
Recall that $\Delta \in \mathcal{L}(\mathcal{S}(\R^{d}),\mathcal{S}(\R^{d}))$ and the heat semigroup is equicontinuous hence a $(C_{0},1)$-semigroup. 

Let $G \in \mathcal{L}(\mathcal{S}(\R^{d}),\mathcal{S}(\R^{d}))$. If we take  $R(t,\omega,f)=G'$ for every $t \geq 0$, $\omega \in \Omega$, $f \in \mathcal{S}'(\R^{d})$, then $R$ satisfies  \eqref{eqSquareMomenCoefROnDrif} and clearly we have $\int_{0}^{t} \int_{\Phi'} R(s,f) L(ds,df)=G'L_{t}$. Hence, as  explained in Example \ref{examSPDEAsLevyIto} (with $A(t)=\Delta$), the stochastic evolution equation
\begin{equation}\label{eqHeatSEELevy}
d Y_{t}=\Delta Y_{t}dt+G' dL_{t}, \quad t \geq 0,
\end{equation}
with initial condition $Y_{0}=\xi$, has a unique weak solution $X=(X_{t}: t \geq 0)$ which corresponds to its mild solution. Moreover, $X$ is a  L\'evy-It\^{o} process satisfying 
$$ X_{t}=\xi +\int_{0}^{t} \Delta X_{s} ds + G'L_{t}.$$
A particularity of this example is that by Corollary 5.8 in \cite{FonsecaMora:SPDELevy}, the process $X$ has a unique c\`adl\`ag version $(Z_{t}: t \geq 0)$ satisfying for every $t \geq 0$, $\psi \in \Psi$,  
$$ \inner{Z_{t}}{\psi}= \inner{\xi}{S(t)\psi}+\int_{0}^{t} \inner{L_{s}}{G \Delta S(t-s)\psi)} \, ds + \inner{L_{t}}{G \psi}, \quad \forall \psi \in \mathcal{S}(\R^{d}).$$
\end{example}

\begin{example}
 Let $\phi_{1}, \phi_{2}, \dots, \phi_{n} \in \mathcal{S}(\mathbb{R}^d)$ and let $g \in \mathcal{S}'(\mathbb{R}^d)$. Let $h:[0,\infty) \times \R^{n} \rightarrow \R$ be a function of class $C^{1,2}$ and for every $f \in \mathcal{S}'(\mathbb{R}^d)$ let $\overline{\phi}(f) = (\langle f, \phi_{1} \rangle, \dots, \langle f, \phi_{n} \rangle) \in \mathbb{R}^n$.  Define $F: [0, \infty) \times \mathcal{S}'(\mathbb{R}^d) \rightarrow \mathcal{S}'(\mathbb{R}^d)$  by the prescription 
    $$ F(t,f)=h(t,\overline{\phi}(f))g.$$
For a fixed $f \in \mathcal{S}'(\mathbb{R}^d)$, the time derivative $F_t(t,f): \mathbb{R} \to \mathcal{S}'(\mathbb{R}^d)$ is
$$F_{t}(t,f) = \partial_t h(t, \overline{\phi}(f))g.$$
Now we calculate the first and second derivatives with respect to the spatial variable $f$. The first Fréchet derivative $F_x(t,f) \in \mathcal{L}(\mathcal{S}'(\mathbb{R}^d), \mathcal{S}'(\mathbb{R}^d))$ is given by 
$$F_x(t,f)(u) = \sum_{i=1}^{n} \partial_{i+1} h(t,\overline{\phi}(f)) \langle u, \phi_i \rangle g,$$
where  $\partial_{i+1} h$ denotes the partial derivative of $h$ with respect to its $(i+1)$-th argument (matching the coordinate $x_i = \langle f, \phi_i \rangle$).

The second Fréchet derivative $F_{xx}(t,f) \in \mbox{Bil}(\mathcal{S}'(\mathbb{R}^d) , \mathcal{S}'(\mathbb{R}^d);  \mathcal{S}'(\mathbb{R}^d)$ is given by
$$F_{xx}(t,f)(u, v) = \sum_{i=1}^{n} \sum_{j=1}^{n} \partial_{j+1}\partial_{i+1} h(t, \overline{\phi}(f)) \langle u, \phi_i \rangle \langle v, \phi_j \rangle g.$$
Let $(L_{t}: t \geq 0)$ be a $\mathcal{S}'(\mathbb{R}^d)$-valued L\'evy process with L\'evy-It\^{o} decomposition \eqref{eqLevyItoDecomposition}. Let $\overline{0}=(0,\dots, 0)$. If we apply It\^{o}'s formula to $X_{t}=L_{t}$ (that is, to the L\'evy-It\^{o} process with $\xi=0$,   $B(t,\omega)=0$ and $R(t,\omega,f)=I$), we get

 \begin{flalign*}
& h(t,\overline{\phi}(L_{t}))g 
  = h(0,\overline{0})g + \left(\int_{0}^{t} \partial_s h(s, \overline{\phi}(L_{s-})) ds \right) g  \\
& + \left( \sum_{i=1}^{n} \int_0^t  \partial_{i+1} h(s,\overline{\phi}(L_{s-})) \langle dW_s, \phi_i \rangle \right) g    +   \left( \sum_{i=1}^{n} \int_0^t  \partial_{i+1} h(s,\overline{\phi}(L_{s-})) \langle \goth{m}, \phi_i \rangle ds \right) g \\
&  + \frac{1}{2} \left( \sum_{i=1}^{n} \sum_{j=1}^{n} \widetilde{\mathcal{Q}}(\phi_{i},\phi_{j}) \int_{0}^{t}\partial_{j+1}\partial_{i+1} h(s, \overline{\phi}(L_{s-})) ds \right) g \\ 
   & + \left( \int_{0}^{t}\!\! \int_{B_{\rho'}(1)^c} \left[ h(t,\overline{\phi}(L_{t-}+f))-h(t,\overline{\phi}(L_{t-}))  \right] \, N(ds,df) \right) g \\
   &  +  \left( \int_{0}^{t}\!\! \int_{B_{\rho'}(1)} \left[ h(s,\overline{\phi}(L_{s-}+f))-h(s,\overline{\phi}(L_{s-}))  \right] \, \tilde{N}(ds,df) \right)g \\
   &  +  \left(  \int_{0}^{t}\!\! \int_{B_{\rho'}(1)} \left[ h(s,\overline{\phi}(L_{s-}+f))-h(s,\overline{\phi}(L_{s-})) - 
   \sum_{i=1}^{n} \partial_{i+1} h(s,\overline{\phi}(L_{s-})) \langle f, \phi_i \rangle
   \right] \,\nu(df)ds \right)g.
\end{flalign*}
In the calculation of the trace (Definition \ref{defiTrace}) we have used the fact that $q^{c}_{r,0}=\mathcal{Q}$ (with $\widetilde{\mathcal{Q}}$ its associated bilinear form) and $q^{c}_{r,f}=0$ for $f \in B_{\rho'}(1) \setminus \{ 0\}$.
\end{example}

\smallskip

\noindent \textbf{Funding} This work was partially supported by The University of Costa Rica through the grant ``C6163- Ecuaciones diferenciales estoc\'{a}sticas en espacios de Hilbert''. 


\smallskip

\noindent \textbf{Data Availability.} Data sharing not applicable to this article as no data sets were generated or analyzed during the current study.


\noindent \textbf{Conflict of interest} The authors have no conflicts of interest to declare that are relevant to the content of this article.




\end{document}